\documentclass[12pt,american]{amsart}

\usepackage[T1]{fontenc}
\usepackage[latin9]{inputenc}
\usepackage{geometry}
\usepackage{multicol}
\usepackage{multirow}
\usepackage{amstext}
\usepackage{amsthm}
\usepackage{amssymb}
\usepackage{graphicx}
\usepackage{tikz}
\usepackage{caption}
\usepackage{xcolor}
\definecolor{darkblue}{rgb}{0, 0, .6}
\definecolor{grey}{rgb}{.8, .8, .8}

\makeatletter
\providecommand{\tabularnewline}{\\}

\numberwithin{equation}{section}
\numberwithin{figure}{section}
\numberwithin{table}{section}
\theoremstyle{plain}
\newtheorem{thm}{\protect\theoremname}[section]
\theoremstyle{definition}
\newtheorem{example}[thm]{\protect\examplename}
\theoremstyle{plain}
\newtheorem{prop}[thm]{\protect\propositionname}
\theoremstyle{plain}

\theoremstyle{remark}
\newtheorem{rem}[thm]{\protect\remarkname}
\theoremstyle{plain}
\newtheorem{corollary}[thm]{\protect\corollaryname}

\usepackage[fontsize=14pt]{scrextend}
\usepackage{graphicx,cases,mathtools}

\newtheoremstyle{plain}
  {\topsep}   % ABOVESPACE
  {\topsep plus 20pt minus 5pt}   % BELOWSPACE
  {}  % BODYFONT
  {0pt}       % INDENT (empty value is the same as 0pt)
  {\bfseries} % HEADFONT
  {.}         % HEADPUNCT
  {5pt plus 1pt minus 1pt} % HEADSPACE
  {}          % CUSTOM-HEAD-SPEC

\usepackage{enumitem}
\setlist{itemsep=0pt,topsep=0pt,parsep=1pt,partopsep=0pt}

\usepackage{pgf,tikz}

\usetikzlibrary{arrows}

\usetikzlibrary{shapes}

\usetikzlibrary{patterns}

\renewcommand*\env@cases[1][1]{%
  \let\@ifnextchar\new@ifnextchar
  \left\lbrace
  \def\arraystretch{#1}%
  \array{@{}l@{\quad}l@{}}%
}

\makeatother

\usepackage{babel}
\providecommand{\conjecturename}{Conjecture}
\providecommand{\examplename}{Example}
\providecommand{\propositionname}{Proposition}
\providecommand{\remarkname}{Remark}
\providecommand{\theoremname}{Theorem}
\providecommand{\corollaryname}{Corollary}

\global\long\def\GEN{\text{GEN}}%
\global\long\def\DNG{\text{DNG}}%
\newcommand{\cupdot}{\mathop{\dot{\cup}}}
\usepackage{relsize}
\newcommand{\boxprod}{\mathop{\textstyle\mathsmaller{\square}}}

\DeclareMathOperator{\Wd}{Wd}

\global\long\def\Opt{\text{Opt}}%

\global\long\def\mex{\text{mex}}%

\global\long\def\nim{\text{nim}}%

\global\long\def\pty{\text{pty}}%

\global\long\def\type{\text{type}}%

\global\long\def\genmaxng{\mathcal{N}}%
\definecolor{rred}{rgb}{0.9, 0.17, 0.31}
\definecolor{darkred}{cmyk}{0,1,1,.3}
\definecolor{rose}{cmyk}{0,1.00,.20,0}
\definecolor{darkblue}{RGB}{0,0,255}
\definecolor{turq}{RGB}{72,209,204}
\definecolor{naublue}{cmyk}{1,.72,0,.32}
\definecolor{mediterranean}{cmyk}{.67,0,.08,.3}
\definecolor{butterfly}{cmyk}{.95,.59,0,.10}
\definecolor{icyblue}{cmyk}{.84,.25,0,.06}
\definecolor{orange2}{RGB}{255,100,0}
\definecolor{orange}{RGB}{255,102,0}
\definecolor{nectarine}{cmyk}{0,0.70,1.00,0}
\definecolor{purple}{RGB}{153,51,255}
\definecolor{darkorchid}{cmyk}{.6,.9,0,.05}
\definecolor{purple2}{RGB}{159,51,250}
\definecolor{gray}{RGB}{220,220,220}
\definecolor{manatee}{rgb}{0.59, 0.6, 0.67}
\definecolor{naugreen}{cmyk}{.43,0,.34,.38}
\definecolor{springgreen}{cmyk}{1.00,0,.70,.02}
\definecolor{ggreen}{RGB}{0,153,0}

\usetikzlibrary{arrows,backgrounds,calc,trees}

\pgfdeclarelayer{background}
\pgfsetlayers{background,main}

\newcommand{\convexpath}[2]{
[   
    create hullnodes/.code={
        \global\edef\namelist{#1}
        \foreach [count=\counter] \nodename in \namelist {
            \global\edef\numberofnodes{\counter}
            \node at (\nodename) [draw=none,name=hullnode\counter] {};
        }
        \node at (hullnode\numberofnodes) [name=hullnode0,draw=none] {};
        \pgfmathtruncatemacro\lastnumber{\numberofnodes+1}
        \node at (hullnode1) [name=hullnode\lastnumber,draw=none] {};
    },
    create hullnodes
]
($(hullnode1)!#2!-90:(hullnode0)$)
\foreach [
    evaluate=\currentnode as \previousnode using \currentnode-1,
    evaluate=\currentnode as \nextnode using \currentnode+1
    ] \currentnode in {1,...,\numberofnodes} {
  let
    \p1 = ($(hullnode\currentnode)!#2!-90:(hullnode\previousnode)$),
    \p2 = ($(hullnode\currentnode)!#2!90:(hullnode\nextnode)$),
    \p3 = ($(\p1) - (hullnode\currentnode)$),
    \n1 = {atan2(\y3,\x3)},
    \p4 = ($(\p2) - (hullnode\currentnode)$),
    \n2 = {atan2(\y4,\x4)},
    \n{delta} = {-Mod(\n1-\n2,360)}
  in 
    {-- (\p1) arc[start angle=\n1, delta angle=\n{delta}, radius=#2] -- (\p2)}
}
-- cycle
}

\tikzset{
    my box/.style = {
        , line cap = round
        , line join = round
    }
}

\tikzstyle{vert} = [circle, draw, fill=grey!80,inner sep=0pt, minimum size=5mm]
\tikzstyle{small vert} = [circle, draw, fill=grey,inner sep=0pt, minimum size=3mm]
\tikzstyle{b} = [draw, very thick, black,-]
\tikzstyle{a} = [draw, black,-stealth]
\begin{document}

\title{Impartial geodetic building games on graphs}

\author{Bret J.~Benesh}
\address{
Department of Mathematics,
College of Saint Benedict and Saint John's University,
37 College Avenue South,
Saint Joseph, MN 56374-5011, USA
}
\email{bbenesh@csbsju.edu}

\author{Dana C.~Ernst}
\address{
Department of Mathematics and Statistics,
Northern Arizona University PO Box 5717,
Flagstaff, AZ 86011-5717, USA
}
\email{Dana.Ernst@nau.edu, Nandor.Sieben@nau.edu}

\author{Marie Meyer}
\address{
Department of Engineering, Computing, and Mathematical Sciences,
Lewis University,
1 University Pkwy,
Romeoville, IL 60446, USA
}
\email{mmeyer2@lewisu.edu}

\author{Sarah K.~Salmon}
\address{Department of Mathematics,
University of Colorado Boulder
Campus Box 395,
2300 Colorado Avenue,
Boulder, CO 80309, USA}
\email{Sarah.Salmon@colorado.edu}

\author{N\'andor Sieben}

\thanks{Date: \the\month/\the\day/\the\year}



%---------------%
\begin{abstract}
%---------------%
A subset of the vertex set of a graph is geodetically convex if it contains every vertex on any shortest path between two elements of the set. The convex hull of a set of vertices is the smallest convex set containing the set. 
We study variations of two games introduced by Buckley and Harary, where two players take turns selecting previously-unselected vertices of a graph until the convex hull of the jointly-selected vertices becomes too large. The last player to move is the winner. The achievement game ends when the convex hull contains every vertex. In the avoidance game, the convex hull is not allowed to contain every vertex.  We determine the nim-value of these games for several graph families.
\end{abstract}

\maketitle

%---------------%
\section{Introduction}
%---------------%

A \emph{geodesic} on a finite graph is a shortest path between two vertices. The \emph{geodetic closure} of a set $S$ of vertices is the collection of vertices contained on geodesics between pairs of vertices of $S$. Harary~\cite{HARARY1984323} and subsequently Buckley and Harary~\cite{BuckleyHarary86} introduced two impartial geodetic games on graphs. In both games, two players alternately take turns selecting previously-unselected vertices, and the geodetic closure of the jointly-selected vertices is computed after each turn.  
The first player who causes the geodetic closure to equal the entire vertex set wins their achievement game.  The first player who cannot select a vertex without causing the geodetic closure to equal the entire vertex set loses their avoidance game. The outcomes of both games were studied for some of the more familiar graphs, including cycles, complete graphs, wheel graphs, generalized wheel graphs (corrected in~\cite{Necascova}), complete bipartite graphs, hypercubes, and the Petersen graph in~\cite{BuckleyHarary86}. Haynes, Henning, and Tiller~\cite{HaynesHenningTiller} studied both the achievement and avoidance games for complete multipartite graphs, coronas, complete block graphs, and split graphs. Other geodetic games on graphs were studied in~\cite{BuckleyHarary85,FraenkelHarary89,Wang17}.

In this paper, we investigate a variation of the games studied in~\cite{BuckleyHarary86,HaynesHenningTiller} that uses convex hulls instead of geodetic closures.  Our variation is analogous to games played on groups studied in ~\cite{anderson.harary:achievement,  
BeneshErnstSiebenGENSpectrum,
BeneshErnstSiebenSymAlt,
BeneshErnstSiebenGeneralizedDihedral,
brandenburg:algebraicGames,
ErnstSieben}.
An intuitive description of the difference between these operators is that the convex hull iteratively applies the geodetic closure until the set stabilizes. See Figure~\ref{fig:sample-game} for an example of the achievement game. The convex hull yields a closure system, whereas the geodetic closure is not idempotent. 
We show that the convex hull and geodetic closure games are different on hypercube graphs, complete multipartite graphs, and generalized wheel graphs. The two variations of the games are the same on all other graph families we consider.

\begin{figure}
\begin{tabular}{ccccc}
\begin{tikzpicture}[scale=1.1,auto]
\filldraw[purple,opacity=0.3] (0,1) circle (7pt);
\node (1) at (0,0) [vert]  {\scriptsize $v_2$};
\node (2) at (0,1) [vert,fill=purple!60] {\scriptsize $v_1$};
\node (3) at (1,-.5) [vert] {\scriptsize $v_3$};
\node (4) at (1,.5) [vert] {\scriptsize $v_4$};
\node (5) at (1,1.5) [vert] {\scriptsize $v_5$};
\path [b] (1) to (3);
\path [b] (1) to (4);
\path [b] (1) to (5);
\path [b] (2) to (3);
\path [b] (2) to (4);
\path [b] (2) to (5);
\begin{pgfonlayer}{background}
\fill[white,opacity=0.3] \convexpath{1,2,5,3}{9pt};
\end{pgfonlayer}
\end{tikzpicture}
&
\raisebox{1.35cm}{
\begin{tikzpicture}[scale=1.1,auto]
\path [a] (0,0) to (.75,0);
\end{tikzpicture}
}
&
\begin{tikzpicture}[scale=1.1,auto]
\node (1) at (0,0) [vert]  {\scriptsize $v_2$};
\node (2) at (0,1) [vert,fill=purple!60] {\scriptsize $v_1$};
\node (3) at (1,-.5) [vert] {\scriptsize $v_3$};
\node (4) at (1,.5) [vert,fill=purple!60] {\scriptsize $v_4$};
\node (5) at (1,1.5) [vert] {\scriptsize $v_5$};
\path [b] (1) to (3);
\path [b] (1) to (4);
\path [b] (1) to (5);
\path [b] (2) to (3);
\path [b] (2) to (4);
\path [b] (2) to (5);
\path [b] (2) to (4);
\begin{pgfonlayer}{background}
\fill[purple,opacity=0.3] \convexpath{2,4}{9pt};
\end{pgfonlayer}
\begin{pgfonlayer}{background}
\fill[white,opacity=0.3] \convexpath{1,2,5,3}{9pt};
\end{pgfonlayer}
\end{tikzpicture}
&
\raisebox{1.35cm}{
\begin{tikzpicture}[scale=1.1,auto]
\path [a] (0,0) to (.75,0);
\end{tikzpicture}
}
&
\begin{tikzpicture}[scale=1.1,auto]
\node (1) at (0,0) [vert]  {\scriptsize $v_2$};
\node (2) at (0,1) [vert,fill=purple!60] {\scriptsize $v_1$};
\node (3) at (1,-.5) [vert,fill=purple!60] {\scriptsize $v_3$};
\node (4) at (1,.5) [vert,fill=purple!60] {\scriptsize $v_4$};
\node (5) at (1,1.5) [vert] {\scriptsize $v_5$};
\path [b,dashed] (1) to (3);
\path [b,dashed] (1) to (4);
\path [b,densely dotted] (1) to (5);
\path [b] (2) to (3);
\path [b] (2) to (4);
\path [b,densely dotted] (2) to (5);
\begin{pgfonlayer}{background}
\fill[purple,opacity=0.3] \convexpath{1,2,5,3}{9pt};
\end{pgfonlayer}
\end{tikzpicture}
\end{tabular}
\caption{
\label{fig:sample-game} 
A sample play of the achievement game on $K_{2,3}$. The first player wins after the moves $v_1$, $v_4$, and $v_3$ indicated by dark purple vertices. The light purple cloud indicates the convex hull of the set of moves. After $v_3$ is selected, $v_2$ must be in the convex hull because it lies along the dashed geodesic  between $v_3$ and $v_4$. As a consequence, $v_5$ is also in the convex hull since it lies along the dotted geodesic between $v_1$ and $v_2$.}
\end{figure}

Whereas~\cite{BuckleyHarary85,BuckleyHarary86,FraenkelHarary89,HaynesHenningTiller} only determine the outcome of their variants, we compute the nim-values, and hence the outcomes, of our games. One consequence of our work is that we have simultaneously determined the nim-values for the geodetic closure variants in instances where the convex hull games and geodetic closure games are the same.

To determine nim-values for our games, we will often follow the approach introduced in~\cite{ErnstSieben} and generalized in~\cite{SiebenHypergraph}, which use structure diagrams for studying the nim-values. Loosely speaking, a structure diagram is a quotient of the game by an equivalence relation called structure equivalence, which respects the nim-values of the positions of the game and drastically simplifies calculation of nim-values. 

The structure of the paper is as follows.  We start with some preliminaries on impartial games and geodetic convexity, and we then define our games.  We recall the necessary structure theory and present results about the effect of graph operations such as unions and products on structure equivalence.  Next, we compute the nim-values for the games on graphs with a unique minimal generating set which includes complete split graphs, corona graphs, and block graphs.  We then compute the nim-values for the games on cycle graphs, hypercube graphs, grid graphs, complete multipartite graphs, wheel graphs, and generalized wheel graphs.  The paper concludes with a summary of our results and some open questions.

%---------------%
\section{Preliminaries\label{sec:Preliminaries}}
%---------------%

\subsection{Notation}
If $f:X\to Y$ and $A\subseteq X$, then we often use the standard
$f(A):=\{f(a)\mid a\in A\}$ notation for the image of $A$. 
As a special case, for a family $\mathcal{A}$ of subsets of $X$ we define
\[
\complement(\mathcal{A}):=\{A^c\mid A\in\mathcal{A}\},
\]
where $A^c=X\setminus A$ is the complement of $A$.
We define the \emph{parity} of the integer $k$ as $\pty(k):=k\!\mod2$. The cardinality of a set $A$ is denoted by $|A|$, and we write $\pty(A):=\pty(|A|)$ for the \emph{parity of a set}. We use the notation $\equiv_m$ to denote congruence modulo $m$.

\subsection{Graphs}
Our general reference for basic graph theory is \cite{PelayoBook}.
A \emph{graph} is an ordered pair $G=(V,E)$, where $V$ is a finite nonempty set of vertices and $E\subseteq {\binom{V}{2}}$ is the set of \emph{edges}, consisting of 2-element subsets of $V$. Two vertices in an edge are called \emph{adjacent}. A sequence $v_0,v_1,\ldots, v_k$ of distinct vertices is called a \emph{path} of \emph{length} $k$ between $v_0$ and $v_k$ if all consecutive vertices are adjacent.

A \emph{digraph} is an ordered pair $G=(V,E)$, where $V$ is a finite nonempty set of vertices and $E\subseteq V\times V$ is the set of \emph{directed edges}.

An \emph{induced subgraph} of a graph $G$ is the graph formed from a subset $S\subseteq V$ and all of the edges of $G$ connecting vertices in $S$. An induced subdigraph is defined analogously.

%---------------%
\subsection{Impartial games}
%---------------%

We recall the basic terminology of impartial combinatorial games. Our general references for the subject are \cite{albert2007lessons,SiegelBook}.

A \emph{finite impartial game} $\mathsf{G}$ is a finite digraph with a unique source but no infinite directed walks. Each vertex is called a \emph{position} while the unique source is referred to as the \emph{starting position}. The elements of the set $\Opt(P)$ of out-neighbors of a position $P$ are called the \emph{options} of $P$. A position $P$ is called \emph{terminal} if $\Opt(P)=\emptyset$. We say that $Q$ is a \emph{subposition} of $P$ if there is a directed walk from $P$ to $Q$.

During a \emph{play} of the game $\mathsf{G}$, two players take turns replacing the current position with one of its options. At the beginning of the play the current position is the starting position. The game ends when the current position becomes a terminal position. The last player to move is the winner of that play. Thus, a play is essentially a directed walk from the starting position to one of the terminal positions.

One can verify that each position of $\mathsf{G}$ is a subposition of the starting position. Since $\mathsf{G}$ has no infinite walks, no position is a proper subposition of itself. Thus, a play is always finite.

Each position $P$ of a game $\mathsf{G}$ determines a game $\mathsf{G}_P$ which is the subdigraph of $\mathsf{G}$ induced by the subpositions of $P$. If $P$ is the initial position, then $\mathsf{G}_P$ is of course $\mathsf{G}$. Accordingly, we define $\Opt(\mathsf{G}_P):=\{\mathsf{G}_Q \mid Q\in\Opt(P)\}$. 
We can think of a play as a process where the players replace the game with a smaller subgame in each turn.

The \emph{minimum excludant} $\mex(A)$ of a set $A$ of nonnegative
integers is the smallest nonnegative integer that is not in $A$.
The \emph{nim-value} (or \emph{Grundy value}) $\nim(P)$ of a position $P$ of a game is defined
recursively as the minimum excludant of the nim-values of the options
of $P$. That is,
\[
\nim(P):=\mex(\nim(\Opt(P))).
\]
Note that a terminal position $P$ has nim-value $\nim(P)=\mex(\nim(\emptyset))=\mex(\emptyset)=0$. The nim-value of the game is the nim-value of the starting position.

A position $P$ is \emph{losing} for the player about to move  if $\nim(P)=0$ and \emph{winning} otherwise. In particular, the second player has a winning strategy on a game $\mathsf{G}$ if and only if $\nim(\mathsf{G})=0$. The winning strategy is to always move to an option with nim-value $0$ since this places the opponent into a losing position.  The nim-value is a central object of interest for impartial games.  In particular, the nim-value determines the outcome of the game.  

The (disjunctive) \emph{sum} $\mathsf{G}+\mathsf{H}$ of the games $\mathsf{G}$ and $\mathsf{H}$ is the product digraph $\mathsf{G}\boxprod\mathsf{H}$, formally defined in Subsection~\ref{subsection:productgraphs}. This means that in each turn a player makes a valid move either in game $\mathsf{G}$ or in game $\mathsf{H}$ and the game sum ends when both games end. The nim-value of the sum of two games can be computed as the \emph{nim-sum} 
\[
\nim(\mathsf{G}+\mathsf{H})=\nim(\mathsf{G})\oplus\nim(\mathsf{H}),
\]
which requires binary addition without carry. A consequence of this is that $\nim(\mathsf{G}+\mathsf{H})=0$ if and only if $\nim(\mathsf{G})=\nim(\mathsf{H})$.

The \emph{nimber} $*n$ is the game satisfying $\Opt(*n)=\{*0,*1,\ldots,*(n-1)\}$.
The only terminal position of this game is the position $*0$. Induction
shows that $\nim(*n)=n$.

We say that positions $P$ and $Q$ of $\mathsf{G}$ are \emph{isomorphic} if $\mathsf{G}_P$ and $\mathsf{G}_Q$ are isomorphic. The following result is folklore, and we provide a proof for easy reference.

\begin{prop}\label{equivalent}
If all the options of a position $P$ are isomorphic, then $\nim(P)\in\{0,1\}$.
\end{prop}

\begin{proof}
Every option of $P$ must have the same nim-value, say $k$. So $\nim(\Opt(P))=\{\nim(Q)\mid Q\in\Opt(P)\}=\{k\}$. Thus 
\[
\nim(P)=\mex(\{k\})=\begin{cases}
0, & k>0\\
1, & k=0.
\end{cases}
\]
\end{proof}

Following~\cite{BuckleyHarary86,FraenkelHarary89}, a move in a position where every option is isomorphic is jokingly called \emph{shrewd}. Note that a shrewd move cannot affect the outcome of the game. %\note{FraenkelHarary says this: "In [2], a shrewd move is (jokingly) defined as a move which always results in

\begin{prop} \label{prop:all-same general}
If the number of moves of every play has the same parity $r$, then $\nim(\mathsf{G})=r$.
\end{prop}
\begin{proof}
If $r=0$, then the second player wins $\mathsf{G}$ by making arbitrary legal moves. If $r=1$, then the second player wins $\mathsf{G}+*1$ by making arbitrary legal moves. 
\end{proof}

%---------------%
\subsection{Geodetic convexity}
%---------------%

A comprehensive reference about geodetic convexity is \cite{PelayoBook}. A \emph{geodesic} of a graph $G=(V,E)$ is a path between two vertices with minimum length. 

\begin{example}
Consider the graph shown in Figure~\ref{fig:geodesic}. The paths $v_1,v_2,v_3,v_5$ and $v_1,v_2,v_4,v_5$ of length $3$ are the only geodesics between $v_1$ and $v_5$. The path $v_1,v_2,v_3,v_6,v_5$ is not a geodesic since it has a larger length of $4$. The path $v_6,v_5,v_4$, indicated by dashed blue, is also a geodesic connecting the vertices $v_6$ and $v_4$. 
\end{example}

\begin{figure}[h]
\centering
\begin{tabular}{cc}
\begin{tikzpicture}[scale=1.1,auto]
\node (1) at (0,0) [vert] {\scriptsize $v_1$};
\node (2) at (1,0) [vert] {\scriptsize $v_2$};
\node (3) at (2,.5) [vert] {\scriptsize $v_3$};
\node (4) at (2,-.5) [vert] {\scriptsize $v_4$};
\node (5) at (3,-.5) [vert] {\scriptsize $v_5$};
\node (6) at (3,.5) [vert] {\scriptsize $v_6$};
\path [b,orange,densely dotted] (1) to (2);
\path [b,orange,densely dotted] (2) to (3);
\path [b] (2) to (4);
\path [b,orange,densely dotted] (3) to (5);
\path [b] (3) to (6);
\path [densely dashed,b,blue] (4) to (5);
\path [densely dashed,b,blue] (5) to (6);
%\begin{pgfonlayer}{background}
%\fill[purple,opacity=0.3] \convexpath{1,2,5,3}{7pt};
%\end{pgfonlayer}
\end{tikzpicture}
\end{tabular}
\caption{A graph and two of its geodesics indicated by dotted orange and dashed blue.}
\label{fig:geodesic}
\end{figure}

The \emph{geodetic closure} $I[P]$ of a subset $P$ of $V$ consists of the vertices along the geodesics connecting two vertices in $P$. The geodetic closure is a pre-closure operator but not necessarily a closure operator because it may not be idempotent. To make it a closure operator, we need to iterate the geodetic closure function until the result stabilizes.

The subset $P$ is called \emph{geodetically convex}, or simply convex, if it contains all vertices along the geodesics connecting two vertices of $P$. The \emph{convex hull} $[P]:=\bigcap\{K\mid P\subseteq K, K \text{ is convex} \}$ of $P$ is the smallest convex set containing $P$. The convex hull function $P\mapsto[P]:2^V\to 2^V$ is a closure operator. In particular, $P$ is convex if and only if $[P]=P$. We say that a set $P$ of vertices is \emph{generating} if $[P]=V$. Otherwise, $P$ is called \emph{nongenerating}.

\begin{example}
Consider the complete bipartite graph $K_{2,3}$ depicted in Figure~\ref{fig:K23}. The set $P = \{v_3,v_4\}$ is generating but $I[P] = \{v_1,v_2,v_3,v_4\}$. 
\end{example}

\begin{figure}[h]
\centering
\begin{tabular}{cc}
\begin{tikzpicture}[scale=1,auto]
\node (1) at (0,0) [vert] {\scriptsize $v_1$};
\node (2) at (0,1) [vert] {\scriptsize $v_2$};
\node (3) at (1,-.5) [vert,fill=purple!60] {\scriptsize $v_3$};
\node (4) at (1,.5) [vert,fill=purple!60] {\scriptsize $v_4$};
\node (5) at (1,1.5) [vert] {\scriptsize $v_5$};
\path [b] (1) to (3);
\path [b] (1) to (4);
\path [b] (1) to (5);
\path [b] (2) to (3);
\path [b] (2) to (4);
\path [b] (2) to (5);
\begin{pgfonlayer}{background}
\fill[purple,opacity=0.3] \convexpath{1,2,5,3}{10pt};
\end{pgfonlayer}
\end{tikzpicture}
  \qquad & \qquad
\begin{tikzpicture}[scale=1,auto]
\node (1) at (0,0) [vert] {\scriptsize $v_1$};
\node (2) at (0,1) [vert] {\scriptsize $v_2$};
\node (3) at (1,-.5) [vert,fill=orange!60] {\scriptsize $v_3$};
\node (4) at (1,.5) [vert,fill=orange!60] {\scriptsize $v_4$};
\node (5) at (1,1.5) [vert] {\scriptsize $v_5$};
\path [b] (1) to (3);
\path [b] (1) to (4);
\path [b] (1) to (5);
\path [b] (2) to (3);
\path [b] (2) to (4);
\path [b] (2) to (5);
\begin{pgfonlayer}{background}
\fill[orange,opacity=0.3] \convexpath{1,2,4,3}{10pt};
\end{pgfonlayer}
\end{tikzpicture}
\\
$[P]$ \qquad & \qquad $I[P]$
\end{tabular}
\caption{Convex hull versus geodetic closure for a two-element subset.}
\label{fig:K23}
\end{figure}

The family of maximal nongenerating sets for a graph $G$ is denoted by $\mathcal{N}(G)$, or simply $\mathcal{N}$ if the context is clear, while the family of minimal generating sets is denoted by $\mathcal{G}(G)$ or $\mathcal{G}$. 

\begin{prop}
\label{prop:vertex is contained in max nongen}
Every vertex of a nontrivial graph is contained in some maximal nongenerating set.
\end{prop}

\begin{proof}
Let $v$ be a vertex. Since the graph is nontrivial, $\{v\}$ is nongenerating. Since the vertex set $V$ is generating, there must be a maximal nongenerating set $N$ such that $\{v\}\subseteq N\subset V$.
\end{proof}

\begin{example}
The trivial graph containing only a single vertex $v$ has no nontrivial nongenerating set. That is, $\mathcal{N}=\{\emptyset\}$. We also have $\mathcal{G}=\{\{v\}\}$.
\end{example}

A vertex is called \emph{simplicial} if the subgraph induced by the neighbors of the vertex is a complete graph. 

\begin{prop}
\label{prop:simplicial}
A generating set contains every simplicial vertex.
\end{prop}

\begin{proof}
For a contradiction, suppose a geodesic between vertices $v$ and $w$ contains a simplicial vertex $u$ not in $\{v,w\}$. Then the vertices on this geodesic immediately before and after $u$ are two different neighbors $r$ and $s$ of $u$. However, since $u$ is simplicial, $r$ and $s$ are adjacent. So replacing $r,u,s$ on the path with $r,s$ creates a shorter path between $v$ and $w$, contradicting the choice of the geodesic. This implies that a simplicial vertex $u$ is not contained in the convex hull of a set $S$ unless $u\in S$.  Therefore, a generating set must contain every simplical vertex.
\end{proof}

\begin{prop}
\label{prop:maximal is closed}
If $N\in\mathcal{N}$, then $N$ is convex. 
\end{prop}

\begin{proof}
The nongenerating set  $[N]$ contains $N$. Hence $[N]=N$ since $N$ is maximal nongenerating.
\end{proof}

We define the \emph{Frattini subset} as $\Phi(G):=\bigcap\mathcal{N}(G)$. If the context is clear, we write $\Phi$ in place of $\Phi(G)$. 

\begin{example}\label{ex:Frattini example}
Consider the cycle graph $C_4$ and the path graph $P_4$, which are depicted in Figure~\ref{fig:Frattini example}. The maximal nongenerating subsets of $C_4$ are $\{v_1,v_2\}$, $\{v_2,v_3\}$, $\{v_3,v_4\}$, and $\{v_1,v_4\}$. This implies that $\Phi(C_4)=\emptyset$.   On the other hand, $\Phi(P_4)=\{v_2,v_3\}$ is nonempty since the maximal nongenerating sets of $P_4$ are $\{v_1,v_2,v_3\}$ and $\{v_2,v_3,v_4\}$. Note that the simplicial vertices of $P_4$ form the unique minimal generating set $\{v_1,v_4\}$.
\end{example}

\begin{figure}[h]
\centering
\begin{tabular}{cc}
\begin{tabular}{c}
\begin{tikzpicture}[scale=1,auto]
\node (1) at (0,0) [vert] {\scriptsize $v_1$};
\node (2) at (1,0) [vert] {\scriptsize $v_2$};
\node (3) at (1,1) [vert] {\scriptsize $v_3$};
\node (4) at (0,1) [vert] {\scriptsize $v_4$};
\path [b] (1) to (2) to (3) to (4) to (1);
\fill[purple,opacity=0] \convexpath{1,3}{7pt};
\end{tikzpicture}\end{tabular}  \qquad & \qquad
\begin{tabular}{c}\begin{tikzpicture}[scale=1,auto]
\node (1) at (0,0) [vert] {\scriptsize $v_1$};
\node (2) at (1,0) [vert,fill=green] {\scriptsize $v_2$};
\node (3) at (2,0) [vert,fill=green] {\scriptsize $v_3$};
\node (4) at (3,0) [vert] {\scriptsize $v_4$};
\path [b] (1) to (2) to (3) to (4);
% \path [b] (1) to (3);
% \fill[green,opacity=0.3] \convexpath{1,3}{7pt};
\end{tikzpicture}\end{tabular}\\
$C_4$ \qquad & \qquad $P_4$
\end{tabular}

\caption{
\label{fig:Frattini example}
The graph $C_4$ with empty Frattini subset and the graph $P_4$ with nonempty Frattini subset highlighted in green.
}
\end{figure}

\begin{prop}
\label{prop:vertexTransitive}
The Frattini subset of a vertex transitive graph is empty.
\end{prop}

\begin{proof}
For a contradiction, suppose that $\Phi$ is nonempty. Then vertex transitivity implies that $\Phi$ contains every vertex. This is a contradiction since the full vertex set is a generating set and so cannot be a subset of a maximal nongenerating set.
\end{proof}

A vertex $v$ is called a \emph{nongenerator} if for all subsets $S$ of vertices $[S]=V$ implies $[S \setminus \{v\}]=V$.   The next result provides a motivation for the name Frattini subset, since it is a special case of the Frattini--Neumann Intersection Theorem \cite{Schmidt}.  We provide a self-contained proof using our terminology.

\begin{prop}\label{useless-elements}
The set of nongenerators is $\Phi$.
\end{prop}

\begin{proof}
Let $v$ be a nongenerator and $N \in \genmaxng$. If $v \not\in N$, then $[N \cup \{v\}]=V$ by maximality of $N$.  Proposition~\ref{prop:maximal is closed} and $v$ being a nongenerator imply the contradiction $N=[N]=V$. So $v\in N$ for all $N \in \genmaxng$, and hence $v \in \Phi$.

Now let $v \in \Phi$ and let $S$ be a generating set. For each $N\in\mathcal{N}$, there exists $s_N\in S \setminus N$ since $N$ is nongenerating. Since $v \in \Phi \subseteq N$, we have $s_N \not=v$. Hence $[\{s_N \mid N \in \mathcal{N}\}]=V$ since $\{s_N \mid N \in \mathcal{N}\}$ is not contained in any maximal nongenerating set. Thus $[S \setminus \{v\}]\supseteq[\{s_N \mid N \in \genmaxng\}]=V$.
\end{proof}

%---------------%
\section{Building games}
%---------------%

For each of the games, we play on a graph $G$ with vertex set $V$ and edge set $E$. Two players take turns selecting previously-unselected vertices of $G$ until the set $P$ of jointly-selected elements satisfy certain conditions. We use the normal play condition, so a player who cannot move loses.

For the achievement game \emph{generate} $\GEN(G)$, the game ends as soon as $[P] = V$, so the player who first creates a generating set wins.  For the avoidance game \emph{do not generate} $\DNG(G)$, each position $P$ must satisfy $[P] \neq V$. The player who cannot select a vertex without creating a generating set loses. We refer to $\GEN(G)$ and $\DNG(G)$ as \emph{building games} as in~\cite{SiebenHypergraph}.

\begin{example}\label{ex:single vertex}
Let $G$ be the trivial graph with a single vertex $v$. The only position of $\DNG(G)$ is the empty set, so the nim-value of this game is $0$. The game for $\GEN(G)$ is $\emptyset\to\{v\}$, and so the nim-value of this game is $1$.
\end{example}

\begin{example}
Let $G$ be a graph with $V=\{u,v\}$, regardless of edge set. The games together with the corresponding nim-values of the positions are shown in Figure~\ref{fig:P2}. Note that $\mathcal{N}=\{\{u\},\{v\}\}$.
\end{example}

\begin{figure}[h]
\begin{tabular}{cc}
\begin{tabular}{c}
\begin{tikzpicture}
\node[label={[label distance=-4pt]left:{\scriptsize $\color{cyan}{*1}$}}] (0) at (1,3) {$\emptyset$};
\node[label={[label distance=-4pt]left:{\scriptsize $\color{cyan}{*0}$}}] (1) at (0,2) {$\{u\}$}; 
\node[label={[label distance=-4pt]right:{\scriptsize $\color{cyan}{*0}$}}] (2) at (2,2) {$\{v\}$}; 
\draw[a] (0) -- (1);
\draw[a] (0) -- (2);
\end{tikzpicture}
\tabularnewline
\end{tabular} & %
\begin{tabular}{c}
\begin{tikzpicture}
\node[label={[label distance=-4pt]left:{\scriptsize $\color{cyan}{*0}$}}] (0) at (1,3) {$\emptyset$};
\node[label={[label distance=-4pt]left:{\scriptsize $\color{cyan}{*1}$}}] (1) at (0,2) {$\{u\}$}; 
\node[label={[label distance=-4pt]right:{\scriptsize $\color{cyan}{*1}$}}] (2) at (2,2) {$\{v\}$}; 
\node[label={[label distance=-4pt]left:{\scriptsize $\color{cyan}{*0}$}}] (3) at (1,1) {$\{u,v\}$};
\draw[a] (0) -- (1);
\draw[a] (0) -- (2);
\draw[a] (1) -- (3);
\draw[a] (2) -- (3);
\end{tikzpicture}
\end{tabular}\\
$\DNG(G)$ & $\GEN(G)$
\end{tabular}

\caption{\label{fig:P2} 
Games and the corresponding nimbers for a graph $G$ with $V=\{u,v\}$.}
\end{figure}

\begin{example}
\label{ex:W5}
Consider the wheel graph $W_5$. Representative quotient digraphs of $\DNG(W_5)$ and $\GEN(W_5)$ are given in Figure~\ref{fig:DNGW5 and GENW5}. In this quotient, we identified geometrically congruent positions. In both cases, we have labeled positions with their corresponding nim-values. Any position that contains two antipodal noncentral vertices will generate $W_5$, which implies that such subsets must be avoided in $\DNG(W_5)$.
\end{example}

\begin{figure}[h]
\centering
\begin{tabular}{cc}
\begin{tikzpicture}[yscale=.85]
\node[inner sep=0pt,label=left:{\scriptsize $\color{cyan}{*1}$}] (top) at (3.125,5.25) {
\begin{tikzpicture}[scale=1,auto]
\node (1) at (0,0) [small vert] {};
\node (2) at (1,0) [small vert] {};
\node (3) at (1,1) [small vert] {};
\node (4) at (0,1) [small vert] {};
\node (5) at (.5,.5) [small vert] {};
\path [b] (1) to (2) to (3) to (4) to (1);
\path [b] (1) to (5);
\path [b] (2) to (5);
\path [b] (3) to (5);
\path [b] (4) to (5);
\begin{pgfonlayer}{background}
\fill[white,opacity=0.3] \convexpath{4,3,2,1}{7pt};
\end{pgfonlayer}
\end{tikzpicture}
};

\node[inner sep=0pt,label=left:{\scriptsize $\color{cyan}{*0}$}] (left1) at (2,3) {
\begin{tikzpicture}[scale=1,auto]
\filldraw[orange,opacity=0.3] (0,1) circle (7pt);
\node (1) at (0,0) [small vert] {};
\node (2) at (1,0) [small vert] {};
\node (3) at (1,1) [small vert] {};
\node (4) at (0,1) [small vert,fill=orange] {};
\node (5) at (.5,.5) [small vert] {};
\path [b] (1) to (2) to (3) to (4) to (1);
\path [b] (1) to (5);
\path [b] (2) to (5);
\path [b] (3) to (5);
\path [b] (4) to (5);
% \begin{pgfonlayer}{background}
% \fill[white,opacity=0.3] \convexpath{4,4}{7pt};
% \end{pgfonlayer}
\begin{pgfonlayer}{background}
\fill[white,opacity=0.3] \convexpath{4,3,2,1}{7pt};
\end{pgfonlayer}
\end{tikzpicture}
};
\node[inner sep=0pt,label=right:{\scriptsize $\color{cyan}{*0}$}] (right1) at (4.25,3) {
\begin{tikzpicture}[scale=1,auto]
\filldraw[orange,opacity=0.3] (.5,.5) circle (7pt);
\node (1) at (0,0) [small vert] {};
\node (2) at (1,0) [small vert] {};
\node (3) at (1,1) [small vert] {};
\node (4) at (0,1) [small vert] {};
\node (5) at (.5,.5) [small vert,fill=orange] {};
\path [b] (1) to (2) to (3) to (4) to (1);
\path [b] (1) to (5);
\path [b] (2) to (5);
\path [b] (3) to (5);
\path [b] (4) to (5);
% \begin{pgfonlayer}{background}
% \fill[white,opacity=0.3] \convexpath{4,4}{7pt};
% \end{pgfonlayer}
\begin{pgfonlayer}{background}
\fill[white,opacity=0.3] \convexpath{4,3,2,1}{7pt};
\end{pgfonlayer}
\end{tikzpicture}
};
\node[inner sep=0pt,label=left:{\scriptsize $\color{cyan}{*1}$}] (left2) at (2,.75) {
\begin{tikzpicture}[scale=1,auto]
\node (1) at (0,0) [small vert,fill=orange] {};
\node (2) at (1,0) [small vert] {};
\node (3) at (1,1) [small vert] {};
\node (4) at (0,1) [small vert,fill=orange] {};
\node (5) at (.5,.5) [small vert] {};
\path [b] (1) to (2) to (3) to (4) to (1);
\path [b] (1) to (5);
\path [b] (2) to (5);
\path [b] (3) to (5);
\path [b] (4) to (5);
\begin{pgfonlayer}{background}
\fill[orange,opacity=0.3] \convexpath{1,4}{7pt};
\end{pgfonlayer}
\begin{pgfonlayer}{background}
\fill[white,opacity=0.3] \convexpath{4,3,2,1}{7pt};
\end{pgfonlayer}
\end{tikzpicture}
};
\node[inner sep=0pt,label=right:{\scriptsize $\color{cyan}{*1}$}] (right2) at (4.25,.75) {
\begin{tikzpicture}[scale=1,auto]
\node (1) at (0,0) [small vert] {};
\node (2) at (1,0) [small vert] {};
\node (3) at (1,1) [small vert] {};
\node (4) at (0,1) [small vert,fill=orange] {};
\node (5) at (.5,.5) [small vert,fill=orange] {};
\path [b] (1) to (2) to (3) to (4) to (1);
\path [b] (1) to (5);
\path [b] (2) to (5);
\path [b] (3) to (5);
\path [b] (4) to (5);
\begin{pgfonlayer}{background}
\fill[orange,opacity=0.3] \convexpath{4,5}{7pt};
\end{pgfonlayer}
\begin{pgfonlayer}{background}
\fill[white,opacity=0.3] \convexpath{4,3,2,1}{7pt};
\end{pgfonlayer}
\end{tikzpicture}
};
\node[inner sep=0pt,label=left:{\scriptsize $\color{cyan}{*0}$}] (left3) at (3.125,-1.5) {
\begin{tikzpicture}[scale=1,auto]
\node (1) at (0,0) [small vert,fill=orange] {};
\node (2) at (1,0) [small vert] {};
\node (3) at (1,1) [small vert] {};
\node (4) at (0,1) [small vert,fill=orange] {};
\node (5) at (.5,.5) [small vert,fill=orange] {};
\path [b] (1) to (2) to (3) to (4) to (1);
\path [b] (1) to (5);
\path [b] (2) to (5);
\path [b] (3) to (5);
\path [b] (4) to (5);
\begin{pgfonlayer}{background}
\fill[orange,opacity=0.3] \convexpath{1,4,5}{7pt};
\end{pgfonlayer}
\begin{pgfonlayer}{background}
\fill[white,opacity=0.3] \convexpath{4,3,2,1}{7pt};
\end{pgfonlayer}
\end{tikzpicture}
};
\path [a] (top) to (left1);
\path [a] (top) to (right1);
\path [a] (left1) to (left2);
\path [a] (right1) to (right2);
\path [a] (left1) to (right2);
\path [a] (left2) to (left3);
\path [a] (right2) to (left3);
\end{tikzpicture} 
\qquad & \qquad
\begin{tikzpicture}[rotate=-90,xscale=.85]
\node[inner sep=0pt,label=left:{\scriptsize $\color{cyan}{*2}$}] (middle1) at (0,0) {
\begin{tikzpicture}[scale=1,auto]
\node (1) at (0,0) [small vert] {};
\node (2) at (1,0) [small vert] {};
\node (3) at (1,1) [small vert] {};
\node (4) at (0,1) [small vert] {};
\node (5) at (.5,.5) [small vert] {};
\path [b] (1) to (2) to (3) to (4) to (1);
\path [b] (1) to (5);
\path [b] (2) to (5);
\path [b] (3) to (5);
\path [b] (4) to (5);
\begin{pgfonlayer}{background}
\fill[white,opacity=0.3] \convexpath{4,3,2,1}{7pt};
\end{pgfonlayer}
\end{tikzpicture}
};
\node[inner sep=0pt,label=left:{\scriptsize $\color{cyan}{*1}$}] (middle2) at (2.25,0) {
\begin{tikzpicture}[scale=1,auto]
\filldraw[purple,opacity=0.3] (0,1) circle (7pt);
\node (1) at (0,0) [small vert] {};
\node (2) at (1,0) [small vert] {};
\node (3) at (1,1) [small vert] {};
\node (4) at (0,1) [small vert,fill=purple] {};
\node (5) at (.5,.5) [small vert] {};
\path [b] (1) to (2) to (3) to (4) to (1);
\path [b] (1) to (5);
\path [b] (2) to (5);
\path [b] (3) to (5);
\path [b] (4) to (5);
% \begin{pgfonlayer}{background}
% \fill[white,opacity=0.3] \convexpath{4,4}{7pt};
% \end{pgfonlayer}
\begin{pgfonlayer}{background}
\fill[white,opacity=0.3] \convexpath{4,3,2,1}{7pt};
\end{pgfonlayer}
\end{tikzpicture}
};
\node[inner sep=0pt,label=right:{\scriptsize $\color{cyan}{*0}$}] (right2) at (2.25,2.25) {
\begin{tikzpicture}[scale=1,auto]
\filldraw[purple,opacity=0.3] (.5,.5) circle (7pt);
\node (1) at (0,0) [small vert] {};
\node (2) at (1,0) [small vert] {};
\node (3) at (1,1) [small vert] {};
\node (4) at (0,1) [small vert] {};
\node (5) at (.5,.5) [small vert,fill=purple] {};
\path [b] (1) to (2) to (3) to (4) to (1);
\path [b] (1) to (5);
\path [b] (2) to (5);
\path [b] (3) to (5);
\path [b] (4) to (5);
% \begin{pgfonlayer}{background}
% \fill[white,opacity=0.3] \convexpath{4,4}{7pt}; 
% \end{pgfonlayer}
\begin{pgfonlayer}{background}
\fill[white,opacity=0.3] \convexpath{4,3,2,1}{7pt};
\end{pgfonlayer}
\end{tikzpicture}
};
\node[inner sep=0pt,label=left:{\scriptsize $\color{cyan}{*0}$}] (left3) at (4.5,-2.25) {
\begin{tikzpicture}[scale=1,auto]
\node (1) at (0,0) [small vert] {};
\node (2) at (1,0) [small vert,fill=purple] {};
\node (3) at (1,1) [small vert] {};
\node (4) at (0,1) [small vert,fill=purple] {};
\node (5) at (.5,.5) [small vert] {};
\path [b] (1) to (2) to (3) to (4) to (1);
\path [b] (1) to (5);
\path [b] (2) to (5);
\path [b] (3) to (5);
\path [b] (4) to (5);
\begin{pgfonlayer}{background}
\fill[purple,opacity=0.3] \convexpath{1,4,3,2}{7pt};
\end{pgfonlayer}
\end{tikzpicture}
};
\node[inner sep=0pt,label=left:{\scriptsize $\color{cyan}{*2}$}] (middle3) at (4.5,0) {
\begin{tikzpicture}[scale=1,auto]
\node (1) at (0,0) [small vert,fill=purple] {};
\node (2) at (1,0) [small vert] {};
\node (3) at (1,1) [small vert] {};
\node (4) at (0,1) [small vert,fill=purple] {};
\node (5) at (.5,.5) [small vert] {};
\path [b] (1) to (2) to (3) to (4) to (1);
\path [b] (1) to (5);
\path [b] (2) to (5);
\path [b] (3) to (5);
\path [b] (4) to (5);
\begin{pgfonlayer}{background}
\fill[purple,opacity=0.3] \convexpath{1,4}{7pt};
\end{pgfonlayer}
\begin{pgfonlayer}{background}
\fill[white,opacity=0.3] \convexpath{4,3,2,1}{7pt};
\end{pgfonlayer}
\end{tikzpicture}
};
\node[inner sep=0pt,label=right:{\scriptsize $\color{cyan}{*2}$}] (right3) at (4.5,2.25) {
\begin{tikzpicture}[scale=1,auto]
\node (1) at (0,0) [small vert] {};
\node (2) at (1,0) [small vert] {};
\node (3) at (1,1) [small vert] {};
\node (4) at (0,1) [small vert,fill=purple] {};
\node (5) at (.5,.5) [small vert,fill=purple] {};
\path [b] (1) to (2) to (3) to (4) to (1);
\path [b] (1) to (5);
\path [b] (2) to (5);
\path [b] (3) to (5);
\path [b] (4) to (5);
\begin{pgfonlayer}{background}
\fill[purple,opacity=0.3] \convexpath{4,5}{7pt};
\end{pgfonlayer}
\begin{pgfonlayer}{background}
\fill[white,opacity=0.3] \convexpath{4,3,2,1}{7pt};
\end{pgfonlayer}
\end{tikzpicture}
};
\node[inner sep=0pt,label=left:{\scriptsize $\color{cyan}{*1}$}] (middle4) at (6.75,0) {
\begin{tikzpicture}[scale=1,auto]
\node (1) at (0,0) [small vert,fill=purple] {};
\node (2) at (1,0) [small vert] {};
\node (3) at (1,1) [small vert] {};
\node (4) at (0,1) [small vert,fill=purple] {};
\node (5) at (.5,.5) [small vert,fill=purple] {};
\path [b] (1) to (2) to (3) to (4) to (1);
\path [b] (1) to (5);
\path [b] (2) to (5);
\path [b] (3) to (5);
\path [b] (4) to (5);
\begin{pgfonlayer}{background}
\fill[purple,opacity=0.3] \convexpath{1,4,5}{7pt};
\end{pgfonlayer}
\begin{pgfonlayer}{background}
\fill[white,opacity=0.3] \convexpath{4,3,2,1}{7pt};
\end{pgfonlayer}
\end{tikzpicture}
};
\node[inner sep=0pt,label=left:{\scriptsize $\color{cyan}{*0}$}] (left4) at (6.75,-2.25) {
\begin{tikzpicture}[scale=1,auto]
\node (1) at (0,0) [small vert,fill=purple] {};
\node (2) at (1,0) [small vert,fill=purple] {};
\node (3) at (1,1) [small vert] {};
\node (4) at (0,1) [small vert,fill=purple] {};
\node (5) at (.5,.5) [small vert] {};
\path [b] (1) to (2) to (3) to (4) to (1);
\path [b] (1) to (5);
\path [b] (2) to (5);
\path [b] (3) to (5);
\path [b] (4) to (5);
\begin{pgfonlayer}{background}
\fill[purple,opacity=0.3] \convexpath{1,4,3,2}{7pt};
\end{pgfonlayer}
\end{tikzpicture}
};
\node[inner sep=0pt,label=right:{\scriptsize $\color{cyan}{*0}$}] (right4) at (6.75,2.25) {
\begin{tikzpicture}[scale=1,auto]
\node (1) at (0,0) [small vert] {};
\node (2) at (1,0) [small vert,fill=purple] {};
\node (3) at (1,1) [small vert] {};
\node (4) at (0,1) [small vert,fill=purple] {};
\node (5) at (.5,.5) [small vert,fill=purple] {};
\path [b] (1) to (2) to (3) to (4) to (1);
\path [b] (1) to (5);
\path [b] (2) to (5);
\path [b] (3) to (5);
\path [b] (4) to (5);
\begin{pgfonlayer}{background}
\fill[purple,opacity=0.3] \convexpath{1,4,3,2}{7pt};
\end{pgfonlayer}
\end{tikzpicture}
};
\node[inner sep=0pt,label=left:{\scriptsize $\color{cyan}{*0}$}] (middle5) at (9,0) {
\begin{tikzpicture}[scale=1,auto]
\node (1) at (0,0) [small vert,fill=purple] {};
\node (2) at (1,0) [small vert] {};
\node (3) at (1,1) [small vert,fill=purple] {};
\node (4) at (0,1) [small vert,fill=purple] {};
\node (5) at (.5,.5) [small vert,fill=purple] {};
\path [b] (1) to (2) to (3) to (4) to (1);
\path [b] (1) to (5);
\path [b] (2) to (5);
\path [b] (3) to (5);
\path [b] (4) to (5);
\begin{pgfonlayer}{background}
\fill[purple,opacity=0.3] \convexpath{1,4,3,2}{7pt};
\end{pgfonlayer}
\end{tikzpicture}
};
\path [a] (middle1) to (middle2);
\path [a] (middle1) to (right2);
\path [a] (middle2) to (left3);
\path [a] (middle2) to (middle3);
\path [a] (middle2) to (right3);
\path [a] (right2) to (right3);
\path [a] (middle3) to (left4);
\path [a] (middle3) to (middle4);
\path [a] (middle4) to (middle5);
\path [a] (right3) to (right4);
\path [a] (right3) to (middle4);
\end{tikzpicture}\\
$\DNG(W_5)$ \qquad & \qquad $\GEN(W_5)$
\end{tabular}

\caption{
\label{fig:DNGW5 and GENW5}
Representative quotients of $\DNG(W_5)$ and $\GEN(W_5)$. Note that the removal of the terminal positions from $\GEN(W_5)$ produces $\DNG(W_5)$.}
\end{figure}

The set of terminal positions of $\DNG(G)$ is $\mathcal{N}$. This characterization provides an opportunity to reformulate Proposition~\ref{prop:all-same general} for the avoidance game in terms of $\mathcal{N}$.

\begin{prop}\label{prop:all-same}
If every set in $\mathcal{N}$ has the same parity $r$ for some graph $G$, then $\nim(\DNG(G))=r$.
\end{prop}

Stating an analogous result for our achievement games in terms of $\mathcal{N}$ is not possible since characterizing the family of terminal positions is more difficult. A full characterization of the nim-value of $\DNG$ played on groups in terms of covering properties of the maximal subgroups is available in \cite{BeneshErnstSiebenDNG}. A similar characterization for the avoidance game $\DNG$ played on graphs using convex hulls is unlikely because the spectrum of possible nim-values for these games seems to be large.

\begin{prop}\label{prop:sum-result}
Assume $|V(H)|$ and $|V(K)|$ are even for the graphs $H$ and $K$, and consider the disjoint union $G:=H\cupdot K$. If $\nim(\DNG(H))=0=\nim(\DNG(K))$, then $\nim(\DNG(G))=0$. If $\nim(\GEN(H))=0=\nim(\GEN(K))$, then $\nim(\GEN(G))=0$.
\end{prop}

\begin{proof}
The second player can win by replying according to the winning strategy in the graph where the first player made a selection. If one of the graphs is already generated, then the second player selects an arbitrary unselected vertex in the graph where the first player's selection came from. Note that unselected vertices in one of the graphs are still available moves until the game finishes, even if that graph is already generated.
\end{proof}

\begin{rem}
\label{rem:Nisall}
A set $P$ of vertices generates if and only if $H\subseteq P$ for some $H\in\mathcal{G}$. On the other hand, $P$ does not generate if and only if $P\subseteq N$ for some $N\in\mathcal{N}$. So $\mathcal{G}$ and $\mathcal{N}$ each completely determine both the avoidance and the achievement games on a graph with a given vertex set. Thus the geodetic closure games of \cite{BuckleyHarary86} and our geodetic convexity games are the same as long as the minimal generating sets or equivalently the maximal nongenerating sets are the same.  See Example~\ref{example:coronacounterexample} for a graph where the geodetic closure and convex hull are different operators, yet the corresponding games are the same.
\end{rem}

%---------------%
\section{Structure equivalence}
%---------------%

For a graph $G$, the games $\GEN(G)$ and $\DNG(G)$, as well as the geodetic closure games developed in \cite{BuckleyHarary86, HaynesHenningTiller,Necascova}, fall into the broader context of building hypergraph games discussed in~\cite[Section~8.1]{SiebenHypergraph}.  
In this section, we recall essential facts from~\cite{ErnstSieben,SiebenHypergraph}. Throughout this section, assume $\mathsf{G}$ is one of the building games $\GEN(G)$ or $\DNG(G)$ for a graph $G$.  

For a subset $P$ of $V$, define $\mathcal{N}_P:=\{N\in\mathcal{N}\mid P\subseteq N\}$ as in~\cite{MccoySieben}. Two positions $P$ and $Q$ of $\mathsf{G}$ are \emph{structure equivalent} if $\mathcal{N}_P = \mathcal{N}_Q$, in which case we write $P \sim Q$. We define the collection of \emph{intersection sets} via
\[
\mathcal{I}:=\left\{\bigcap \mathcal{J}\mid \mathcal{J}\subseteq \mathcal{N}\right\}
\]
and we define $\lceil P \rceil:=\bigcap \mathcal{N}_P$, which is the smallest intersection set containing $P$. Note that $\bigcap \mathcal{J}=V$ if $\mathcal{J}=\emptyset$. Thus, positions $P$ and $Q$ are structure equivalent if and only if $\lceil P \rceil =\lceil Q \rceil$ by~\cite[Proposition~4.2]{SiebenHypergraph}. Note that $\mathcal{I}$ is a closure system and $P\subseteq [P]\subseteq \lceil P\rceil$ for all $P$.

\begin{example}
Consider the path graph $P_4$ of Example~\ref{ex:Frattini example} shown in Figure~\ref{fig:Frattini example}.  We have 
\[
\mathcal{N}=\{\{v_1,v_2,v_3\},\{v_2,v_3,v_4\}\},\qquad 
\mathcal{I}=\mathcal{N}\cup\{\{v_2,v_3\},V\}.
\]
Positions $P:=\{v_1\}$ and $Q:=\{v_1,v_3\}$ are structure equivalent since $\lceil P\rceil=\{v_1,v_2,v_3\}=\lceil Q\rceil$. We also have $\{v_2\}\sim \{v_3\}$ since $\lceil\{v_2\}\rceil=\{v_2,v_3\}=\lceil\{v_3\}\rceil$.
\end{example}

The \emph{structure class} $X_I$ consists of all positions that are structure equivalent to $I\in\mathcal{I}$. Observe that for $I\in \mathcal{I}\setminus \{V\}$, $I$ is the largest element in its structure class.

For the achievement game $\GEN(G)$, $X_V$ consists of all terminal positions. Note that $V$ might not be a game position in $\GEN(G)$, in which case $V \notin X_V$.  On the other hand, $X_V = \emptyset$ for the avoidance game $\DNG(G)$.

We say that a structure class is \emph{terminal} if it contains a terminal position. For $\GEN(G)$, the only terminal structure class is $X_V$. For $\DNG(G)$, the terminal structure classes are of the form $X_N$ for $N \in \mathcal{N}$.

We say that $X_J$ is an \emph{option} of $X_I$ if $X_J\cap \Opt(I)\neq \emptyset$.  We denote the set of options of $X_I$ by $\Opt(X_I)$. Proposition~4.5 from \cite{SiebenHypergraph} implies that if $X_J\in\Opt(X_I)$, then $X_J\cap \Opt(P)\neq \emptyset$ for all $P\in X_I$. That is, structure equivalence is compatible with the option relationship between game positions.

The \emph{type} of the structure class $X_I$ is defined via
\[
\type(X_I) := (\pty(I),\nim_0(X_I),\nim_1(X_I)),
\]
where if $I=V$, then $\nim_0(X_V):=0$ and $\nim_1(X_V):=0$, and if $I\neq V$, then 
\begin{align*}
\nim_{\pty(I)}(X_I)&:=\mex(\nim_{1-\pty(I)}(\Opt(X_I))),\\
\nim_{1-\pty(I)}(X_I)&:=\mex(\nim_{\pty(I)}(\Opt(X_I))\cup \{\nim_{\pty(I)}(X_I)\}).
\end{align*}
The recursive computation of these types is referred to as \emph{type calculus}. 

The next result appears as~\cite[Proposition~4.12]{SiebenHypergraph} and guarantees that the type of a structure class encodes the nim-values of the positions contained in the class. In fact, if $P\sim Q$ and $\pty(P)=\pty(Q)$, then $\nim(P)=\nim(Q)$.

\begin{prop}
For a position $P$ of $\mathsf{G}$, if $P\in X_I$, then $\nim(P) = \nim_{\pty(P)}(X_I)$.
\end{prop}

The \emph{structure digraph} of $\mathsf{G}$ has vertex set $\{X_I \mid I\in \mathcal{I}\}$ and directed edge set $\{(X_I,X_J) \mid X_J \in\Opt(X_I)\}$. We visualize the structure digraph with a \emph{structure diagram} that also shows the type of each structure class. In particular, a vertex $X_I$ is represented by a triangle pointing up or down depending on the parity of $I$. The triangle points down when $\pty(I) = 1$ and points up when $\pty(I) = 0$. The numbers within each triangle are $\nim_0(X_I)$ and $\nim_1(X_I)$.

Instead of studying the full game, the structure diagram of a game allows us to process a potentially smaller quotient digraph. The utility of a structure diagram depends on the size of the structure classes.

The structure diagram for $\mathsf{G}$ can be quite large, making it difficult to gain any intuition about the game. We can make further identifications in the structure diagram to build a \emph{simplified structure diagram} by identifying $X_I$ and $X_J$ if the following conditions hold:
\begin{enumerate}
\item $\pty(I) = \pty(J)$;
\item $\type(\Opt(X_I)) = \type(\Opt(X_J))$;
\item The lengths of the longest directed paths starting at $X_I$ and at $X_J$ are the same.
\end{enumerate}
Note that the first two conditions imply that $\type(X_I)=\type(X_J)$, while the third condition avoids vertical collapsing. In the simplified structure diagram we use shaded triangles if they represent several structure classes. The triangles in a simplified structure diagram are referred to as \emph{simplified structure classes}.
 
Algorithm~4.13 from~\cite{SiebenHypergraph} describes a process for determining the structure diagram for a game $\mathsf{G}$ using the maximal nongenerating sets and type calculus. Note that the Frattini subset, the smallest set in $\mathcal{I}$, is structure equivalent to the empty set, so $\lceil\emptyset\rceil=\Phi$. That is, the unique source of the structure digraph is $X_\Phi$. Hence we can simply read off the nim-value of $\mathsf{G}$ from the structure diagram or the simplified structure diagram by looking at the first entry displayed in the triangle for $X_{\Phi}$ or the second entry of $\type(X_\Phi)$.

\begin{example}
The family of maximal nongenerating sets for the diamond graph $G$, shown in Figure~\ref{fig:diamond2}, is $\mathcal{N}=\{I,J\}$, with $I=\{v_1,v_2,v_3\}$ and $J=\{v_2,v_3,v_4\}$. The family of minimal generating sets is $\mathcal{G}=\{\{v_1,v_4\}\}$. The family of intersection sets is $\mathcal{I}=\{\Phi,I,J,V\}$, where $\Phi=I\cap J=\{v_2,v_3\}$ is the Frattini subset. 

Figure~\ref{fig:diamond2} also highlights the positions in the three structure classes $X_\Phi$, $X_I$, and $X_J$ of $\DNG(G)$ together with the structure diagram and simplified structure diagram. The sets in $X_V$ are not positions of $\DNG(G)$ since these sets are generating. Since $\Opt(X_I)=\emptyset=\Opt(X_J)$ and $\Opt(X_\Phi)=\{X_I,X_J\}$, type calculus shows that
\[
\type(X_I)=(1,1,0)=\type(X_J),\quad \type(X_\Phi)=(0,1,0).
\]
Since the first component of the common type of $X_I$ and $X_J$ is $1$, these structure classes are represented by downward pointing triangles in the structure diagram. The first component of the type of $X_\Phi$ is $0$, so this structure class is represented by an upward pointing triangle. The numbers in these triangles are the second and third components of the types. Since $\pty(I)=1=\pty(J)$ and $\Opt(X_I)=\Opt(X_J)$, structure classes $X_I$ and $X_J$ are identified in the simplified structure diagram. 
This identification is indicated by the shading of the bottom triangle.

Figure~\ref{fig:diamond} shows the structure diagram and the simplified structure diagram for $\GEN(G)$. This structure diagram contains the additional structure class $X_V$ with $\Opt(X_V)=\emptyset$. Now we have $\Opt(X_I)=\{X_V\}=\Opt(X_J)$, and so 
\[
\type(X_V) =(0,0,0),\quad \type(X_I)=(1,2,1)=\type(X_J),\quad \type(X_\Phi)=(0,0,1). 
\]
Structure classes $X_I$ and $X_J$ are again identified in the simplified structure diagram. 

These structure diagrams show that $\nim(\DNG(G))=1$ and $\nim(\GEN(G))=0$. These nim-values are the first numbers in the top triangles corresponding to $X_\Phi$.
\end{example}

\begin{figure}[h]
\begin{tabular}{c}
\begin{tikzpicture}[scale=.8,auto]
\node (1) at (0,1) [ vert] {\scriptsize $v_1$};
\node[vert, fill=green] (2) at (1,2)  {\scriptsize $v_2$};
\node (3) at (1,0) [ vert, fill=green] {\scriptsize $v_3$};
\node (4) at (2,1) [ vert] {\scriptsize $v_4$};
\path [b] (1) to (2) to (4) to (3) to (1);
\path [b] (2) to (3);
\end{tikzpicture}
\end{tabular}
\begin{tabular}{c}
\begin{tikzpicture}[xscale=1.5,yscale=1.2,inner sep=.3mm]k\draw[line width=10mm,yellow!35,rounded corners, line cap=round] (0,1) -- (-.5,2) -- (0,3) -- (.5,2) -- (0,1) -- (0,2);
\draw[line width=12mm,yellow!35,rounded corners, line cap=round] (-1.5,0) -- (-2,1) -- (-1.5,2) -- (-1,1) -- (-1.5,0) -- (-1.5,1);
\draw[line width=12mm,yellow!35,rounded corners, line cap=round] (1.5,0) -- (2,1) -- (1.5,2) -- (1,1) -- (1.5,0) -- (1.5,1);
\node (0) at (0,3) {$\underset{\color{cyan}*1}{\scriptstyle\emptyset}$};
\node (1) at (-1.5,2) {$\underset{\color{cyan}*0}{\scriptstyle\{v_1\}}$};
\node (2) at (-.5,2) {$\underset{\color{cyan}*0}{\scriptstyle\{v_2\}}$};
\node (3) at (.5,2) {$\underset{\color{cyan}*0}{\scriptstyle\{v_3\}}$};
\node (4) at (1.5,2) {$\underset{\color{cyan}*0}{\scriptstyle\{v_4\}}$};
\node (12) at (-2,1) {$\underset{\color{cyan}*1}{\scriptstyle\{v_1,v_2\}}$};
\node (13) at (-1,1) {$\underset{\color{cyan}*1}{\scriptstyle\{v_1,v_3\}}$};
\node (23) at (0,1) {$\underset{\color{cyan}*1}{\scriptstyle\{v_2,v_3\}}$};
\node (24) at (1,1) {$\underset{\color{cyan}*1}{\scriptstyle\{v_2,v_4\}}$};
\node (34) at (2,1) {$\underset{\color{cyan}*1}{\scriptstyle\{v_3,v_4\}}$};
\node (123) at (-1.5,0) {$\underset{\color{cyan}*0}{\scriptstyle\{v_1,v_2,v_3\}}$};
\node (234) at (1.5,0) {$\underset{\color{cyan}*0}{\scriptstyle\{v_2,v_3,v_4\}}$};
\draw[->] (0)--(3);
\draw[->] (0)--(2);
\draw[->] (0)--(1);
\draw[->] (0)--(4);
\draw[->] (1)--(12);
\draw[->] (1)--(13);
\draw[->] (2)--(23);
\draw[->] (2)--(12);
\draw[->] (2)--(24);
\draw[->] (3)--(13);
\draw[->] (3)--(23);
\draw[->] (3)--(34);
\draw[->] (4)--(24);
\draw[->] (4)--(34);
\draw[->] (12)--(123);
\draw[->] (13)--(123);
\draw[->] (23)--(123);
\draw[->] (23)--(234);
\draw[->] (24)--(234);
\draw[->] (34)--(234);
\end{tikzpicture}
\end{tabular}
\begin{tabular}{c}
\includegraphics[scale=.4]{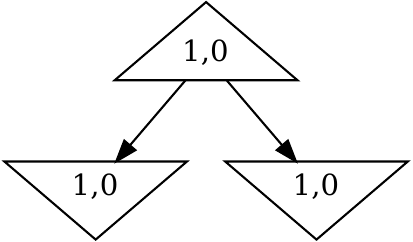} \\
\\
\includegraphics[scale=.4]{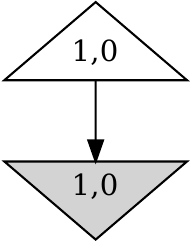}
\end{tabular}
\hfil

\caption{
\label{fig:diamond2}
The diamond graph $G$ with the Frattini subset highlighted in green, the game $\DNG(G)$, and its structure diagrams. The yellow clouds on the top, left, and right indicate the structure classes $X_\Phi$, $X_I$, and $X_J$, respectively. 
}
\end{figure}

\begin{figure}[h]
\includegraphics[scale=.4]{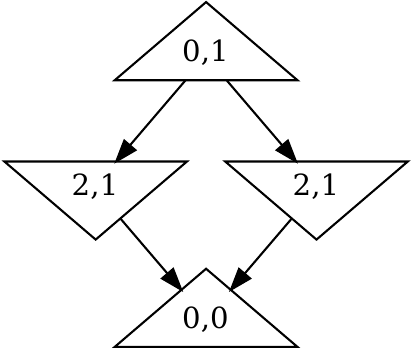}
\qquad
\includegraphics[scale=.4]{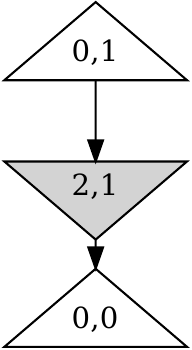}

\caption{
\label{fig:diamond}
The structure diagrams for the game $\GEN(G)$ on the diamond graph $G$.}
\end{figure}

\begin{example}
\label{ex:4-pan}
Figure~\ref{fig:kite} shows the 4-pan graph $G$ and the corresponding structure diagrams
of $\DNG(G)$ and $\GEN(G)$. 
We have
\begin{align*}
\mathcal{N}   &= \{\{v_1,v_2,v_3\},\{v_1,v_2,v_4\},\{v_2,v_3,v_4,v_5\}\} \\
\mathcal{G}   &= \{ \{v_1,v_5\},\{v_1,v_3,v_4\} \}.
\end{align*}
Note that $\lceil \{v_3\} \rceil=\{v_1,v_2,v_3\}\cap \{v_2,v_3,v_4,v_5\}=\{v_2,v_3\}$ while the convex hull is $[\{v_3 \}]=\{v_3\}$. The Frattini subset is $\Phi=\{v_2\}$. The structure diagrams show that $\nim(\DNG(G))=2$ and $\nim(\GEN(G))=0$.
\end{example}

\begin{figure}[h]
\begin{tabular}{ccc}
\multirow[c]{2}{*}{
\begin{tikzpicture}[scale=1,auto,rotate=45]
\node (4) at (0,0) [ vert] {\scriptsize $v_4$};
\node (2) at (1,0) [ vert, fill=green] {\scriptsize $v_2$};
\node (3) at (1,1) [ vert] {\scriptsize $v_3$};
\node (5) at (0,1) [ vert] {\scriptsize $v_5$};
\node (1) at (1.707106,-0.707106) [ vert] {\scriptsize $v_1$};
\path [b] (4) to (2) to (3) to (5) to (4);
\path [b] (1) to (2);
\end{tikzpicture}
} \quad & \quad
\includegraphics[scale=0.4]{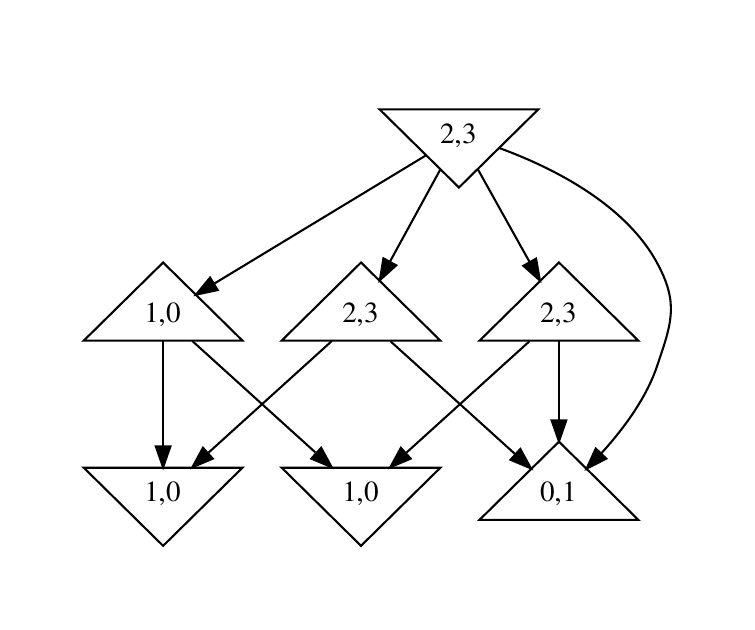} & \includegraphics[scale=0.4]{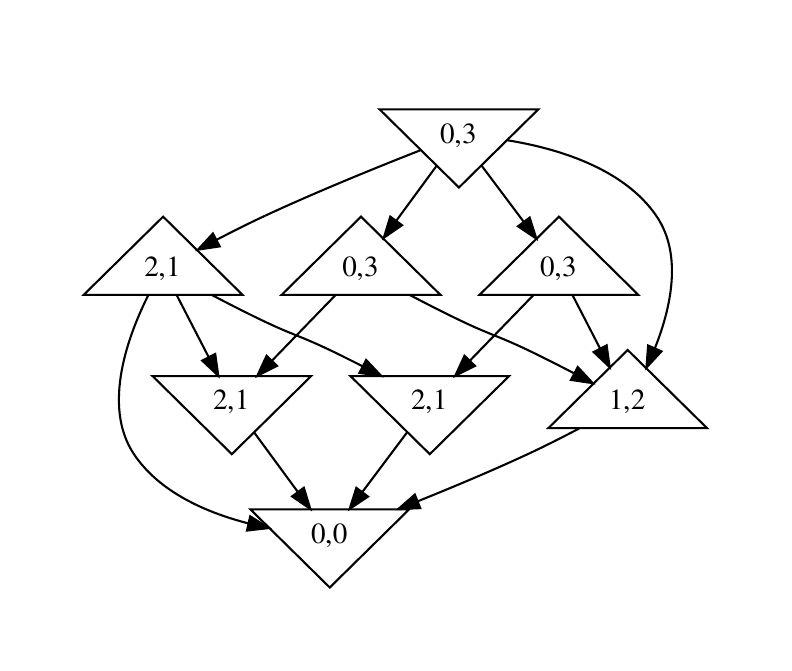}\\
    & \includegraphics[scale=0.4]{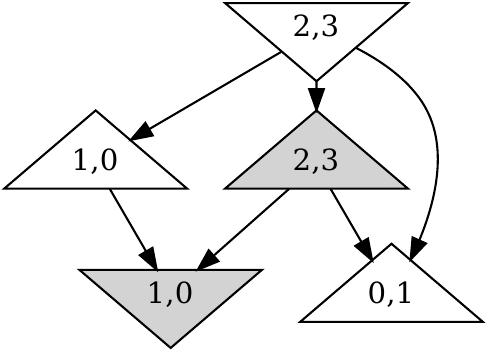} & \includegraphics[scale=0.4]{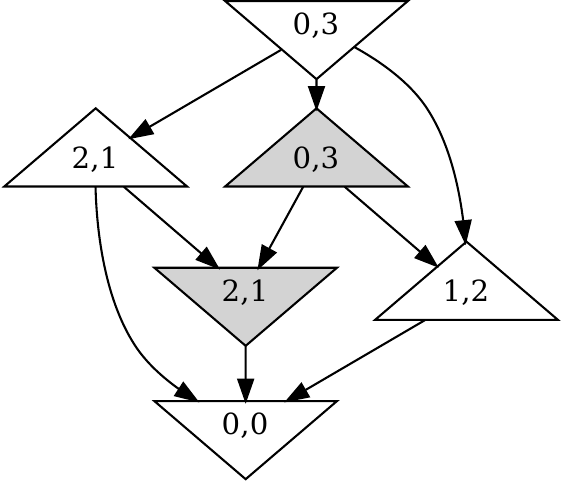}\\
$G$ & $\DNG(G)$ & $\GEN(G)$
\end{tabular}
\caption{\label{fig:kite}
The $4$-pan graph and its structure diagrams. The Frattini subset is highlighted in green.}
\end{figure}

\begin{example}
The Petersen graph $P$ is a diameter 2 graph. Every opening move is shrewd since $P$ is vertex transitive. The structure diagrams in Figure~\ref{fig:Petersen} confirm the predictions of Proposition~\ref{equivalent} by showing that $\nim(\DNG(P))=1$ and $\nim(\GEN(P))=0$. There are two types of maximal nongenerating sets: 5-cycles and a vertex together with its three neighbors. This implies that the Frattini subset is empty, which can also be seen immediately by Proposition~\ref{prop:vertexTransitive}. Note that convex hull of $S=\{v_1,v_3,v_9\}$ is $V$ but the geodetic closure of $S$ is not the full vertex set.  Since the minimal generating sets for our games and those in \cite{BuckleyHarary86} differ, the games are different. \end{example}

\begin{figure}[h]
\centering
\begin{tabular}{ccc}
\quad 
\begin{tikzpicture}[scale=1.1]
\node (1) at (18:2) [ vert] {\scriptsize $v_5$};
\node (2) at (90:2) [ vert] {\scriptsize $v_1$};
\node (3) at (162:2) [ vert] {\scriptsize $v_2$};
\node (4) at (234:2) [ vert] {\scriptsize $v_3$};
\node (5) at (306:2) [ vert] {\scriptsize $v_4$};
\node (11) at (18:1) [ vert] {\scriptsize $v_{10}$};
\node (12) at (90:1) [ vert] {\scriptsize $v_6$};
\node (13) at (162:1) [ vert] {\scriptsize $v_7$};
\node (14) at (234:1) [ vert] {\scriptsize $v_8$};
\node (15) at (306:1) [ vert] {\scriptsize $v_9$};
\path [b] (1) to (2) to (3) to (4) to (5) to (1);
\path [b] (1) to (11) to (13) to (15) to (12) to (14) to (11);
\path [b] (2) to (12);
\path [b] (3) to (13);
\path [b] (4) to (14);
\path [b] (5) to (15);
\end{tikzpicture}
\quad & \quad  \includegraphics[scale=.4]{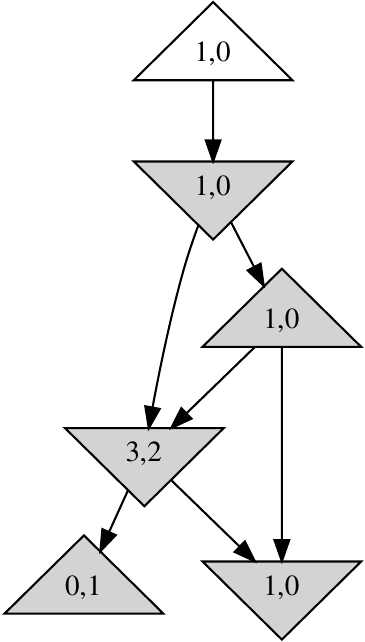} \quad & \quad \includegraphics[scale=.4]{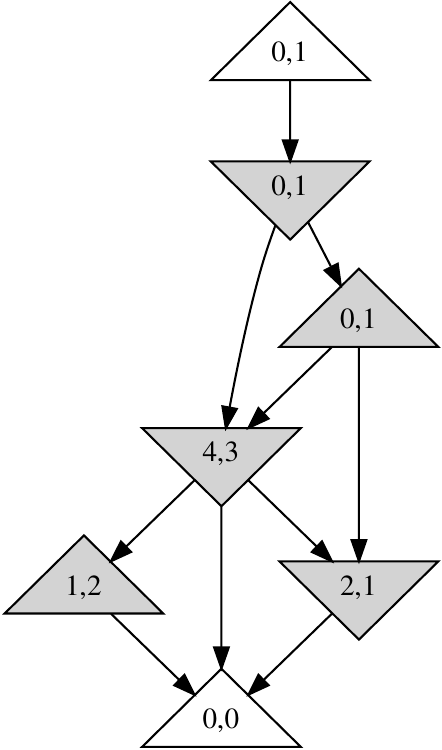}\\
\quad $P$ \quad & \quad $\DNG(P)$ \quad & \quad $\GEN(P)$
\end{tabular}

\caption{
\label{fig:Petersen}
Petersen graph and its simplified structure diagrams.}
\end{figure}

%---------------%
\section{Maximal nongenerating subsets for graph operations}
%---------------%

We saw that the maximal nongenerating sets play an important role in structure equivalence. In this section we develop results to find $\mathcal{N}$ for various graph constructions.

%---------------%
\subsection{Disjoint unions}
%---------------%
For graphs $G$ and $H$, $G\cupdot H$ is the graph that results from taking the disjoint union of $G$ and $H$. The maximal nongenerating subsets for disjoint unions of graphs have the expected form.

\begin{prop}\label{prop:maximal disjoint union}
For graphs $G$ and $H$,
\[
\mathcal{N}(G\cupdot H)=\{N\cup V(H)\mid N\in\mathcal{N}(G)\}\cup \{V(G)\cup N\mid N\in\mathcal{N}(H)\}.
\]
\end{prop}

% \footnote{N: The notation $(\mathcal{N(G)}\lor\{V(H)\})\cup(\{V(G)\}\lor\mathcal{N(H)})$ might be useful.}

\begin{proof}
It is clear that $\mathcal{N}(G\cupdot H)\supseteq \{N \cup V(H) \mid N\in\mathcal{N}(G)\}\cup \{V(G)\cup N\mid N\in\mathcal{N}(H)\}$. Now, suppose $M\in \mathcal{N}(G\cupdot H)$ and choose $v\in V(G\cupdot H)\setminus M$. Then $v$ is an element of exactly one of $V(G)$ or $V(H)$. Without loss of generality, $v\in V(H)$. Since $M$ is a maximal nongenerating set for $V(G\cupdot H)$, $[M\cup\{v\}]=V(G\cupdot H)$. Suppose there exists $u\in V(G)\setminus M$ so that $u\notin [M]$.  But since $v\in V(H)$ and $G$ and $H$ are disjoint, $u\notin [M\cup\{v\}]$, which contradicts $M$ being a maximal nongenerating set.  Thus, it must be the case that $V(G)\subseteq M$.  Certainly, $M\cap V(H)$ must be a maximal nongenerating set for $V(H)$.  This yields the reverse containment.
\end{proof}

The next result follows immediately from Proposition~\ref{prop:maximal disjoint union}.

\begin{corollary}
For graphs $G$ and $H$, $\Phi(G\cupdot H)= \Phi(G) \cup \Phi(H)$.
\end{corollary}

%---------------%
\subsection{Product graphs}\label{subsection:productgraphs}
%---------------%
For graphs $G$ and $H$, the \emph{box product} $G \boxprod H$ is the graph with vertex set $V(G) \times V(H)$ and there is an edge from $(x_1, y_1)$ to $(x_2, y_2)$ if and only if either $x_1=x_2$ with $y_1y_2\in E(H)$, or $y_1=y_2$ with $x_1x_2\in E(G)$.
%Consider the product graph $G\boxprod H$ for graphs $G$ and $H$. 
Define the projection maps $p_G:V(G)\times V(H)\to V(G)$ and $p_H:V(G)\times V(H)\to V(H)$ via $p_G(g,h)=g$ and $p_H(g,h)=h$, respectively. The following result appears as Theorem~3.1 in~\cite{PelayoBook}.

\begin{prop}\label{prop:closure product}
For graphs $G$ and $H$, if $S\subseteq V(G \boxprod H)$, then $[S]=[p_G(S)]\times [p_H(S)]$.
\end{prop}

We can use the previous result to characterize maximal nongenerating sets in product graphs.

\begin{prop}\label{prop:maximal product}
For graphs $G$ and $H$,
\[
\mathcal{N}(G\boxprod H)=\{N\times V(H)\mid N\in \mathcal{N}(G)\}\cup \{V(G)\times N\mid N\in\mathcal{N}(H)\}.
\]
\end{prop} 

\begin{proof}
First, suppose $X\in \mathcal{N}(G\boxprod H)$. By Propositions~\ref{prop:maximal is closed} and~\ref{prop:closure product}, we know
\[
X=[X]=[p_G(X)]\times [p_H(X)].
\]
Without loss of generality, suppose $[p_G(X)]\neq V(G)$. Then $[p_G(X)]\times V(H)$ is a nongenerating set containing $X$ since
\[
[[p_G(X)]\times V(H)]=[p_G(X)]\times V(H)\neq V(G\boxprod H).
\]
Hence $X=[p_G(X)]\times V(H)$ by the maximality of $X$. Thus $[p_H(X)]=V(H)$. It remains to show that $[p_G(X)]$ is a maximal nongenerating set for $V(G)$. If $g\in V(G)\setminus [p_G(X)]$, then $[[p_G(X)]\cup\{g\}]=V(G)$ since 
\[
V(G)\times V(H)=[X\cup\{(g,h)\}]=[p_G(X)\cup\{g\}]\times V(H)\subseteq [[p_G(X)]\cup\{g\}]\times V(H) 
\]
for each $h\in V(H)$. Thus $p_G(X)\in \mathcal{N}(G)$.

Now, without loss of generality, consider $A:=N\times V(H)$ for some $N\in\mathcal{N}(G)$. By Proposition~\ref{prop:closure product}, we see that
\[
[A]=[p_G(A)]\times [p_H(A))]=A\neq V(G\boxprod H).
\]
This shows that $A$ is a nongenerating set for $V(G\boxprod H)$. Let $(g,h)\in V(G\boxprod H)\setminus A$ so that $g\notin N$. Since $N$ is a maximal nongenerating set for $G$, $[N\cup\{g\}]=V(G)$.  By Proposition~\ref{prop:closure product}, we have
\begin{align*}
[A\cup\{(g,h)\}]&=[p_G(A\cup\{(g,h)\})]\times [p_H(A\cup\{(g,h)\})]\\
&=[N\cup\{g\}]\times V(H)\\
&= V(G)\times V(H)\\
&=V(G\boxprod H).
\end{align*}
Thus, $N\times V(H)\in\mathcal{N}(G\boxprod H)$.
\end{proof}

\begin{prop}\label{prop:Frattini product}
For graphs $G$ and $H$, $\Phi(G\boxprod H)=\Phi(G)\times \Phi(H)$.
\end{prop}

\begin{proof}  %https://proofwiki.org/wiki/Cartesian_Product_of_Intersections
Let $N_1,\ldots,N_k$ and $M_1,\ldots, M_l$ be the maximal nongenerating sets for $G$ and $H$, respectively.  Using Proposition~\ref{prop:maximal product} and well-known results concerning Cartesian products and intersections, we see that
%Proof uses Cor 1 and then main theorem at link above.
\begin{align*}
\Phi(G\boxprod H) & = (N_1\times V(H))\cap \cdots \cap (N_k\times V(H)) \cap (V(G)\times M_1)\cap \cdots \cap (V(G)\times M_l)\\
 & = ((N_1\cap \cdots \cap N_k)\times V(H)) \cap (V(G)\times (M_1\cap \cdots \cap M_l))\\
& = (\Phi(G)\times V(H))\cap (V(G)\times \Phi(H))\\
& = (\Phi(G)\cap V(G))\times (V(H)\cap \Phi(H))\\
& = \Phi(G)\times \Phi(H),
\end{align*}
as desired.
\end{proof}

%---------------%
\subsection{1-clique sums}
%---------------%
%\url{https://en.wikipedia.org/wiki/Clique-sum} 

For $k\geq 2$, the \emph{$1$-clique-sum} of the graphs $G_1,\ldots,G_k$ is formed from the disjoint union of the graphs by identifying a vertex from each graph to form a single shared vertex. We also say that $G$ is the $1$-clique-sum at $c$ if $c$ is the shared vertex. Define $\mathcal{N}_c(G_i) := \{N\in\mathcal{N}(G_i) \mid c\in N\}$ and $\Psi_i:=\bigcap\mathcal{N}_c(G_i)$. Note that $\mathcal{N}_c(G_i)\neq \emptyset$ by Proposition~\ref{prop:vertex is contained in max nongen} and $c\in\Psi_i$ for all $i$.

\begin{prop}
\label{prop:1cliqueSum}
If $G$ is the $1$-clique-sum of the nontrivial connected graphs $G_1,\ldots, G_k$ at $c$, then 
\[
\mathcal{N}(G)=\bigcup_{i=1}^k\{ N\cup \bigcup_{j\neq i}V(G_j) \mid N\in \mathcal{N}_c(G_i)\}.
\]
\end{prop}

\begin{proof}
If $N\in\mathcal{N}_c(G_i)$ for $i\in\{1,\ldots,k\}$, then it is clear that $N\cup \bigcup_{j\neq i}V(G_j)$ is a maximal nongenerating set for $G$, which verifies one containment.

Now, let $M\in\mathcal{N}(G)$. Since $M$ is nongenerating, there exists $x\in V(G)\setminus [M]$. Then there exists $i$ such that $x\in V(G_i)$. Certainly, $x\notin [M\cup \bigcup_{j\neq i}V(G_j)]$, and so $M\cup \bigcup_{j\neq i}V(G_j)$ is nongenerating. Since $M$ is maximal nongenerating, this implies that $M=M\cup\bigcup_{j\neq i}V(G_j)$. Define $N:=M\cap V(G_i)$. Since $c\in M$, we also have $c\in N$. Hence $N$ is maximal nongenerating for $G_i$, otherwise $M$ is not maximal nongenerating for $G$. Thus
\[
M= N\cup \bigcup_{j\neq i}V(G_j)
\]
for $N\in\mathcal{N}_c(G_i)$, which shows that $M$ has the desired form.
\end{proof}

\begin{prop}
If $G$ is the $1$-clique-sum of the nontrivial connected graphs $G_1,\ldots, G_k$ at $c$, then 
\[
\Phi(G)=\bigcup_{i=1}^k \Psi_i.
\]
\end{prop}

\begin{proof}
Let $v\in\Phi(G)$. Then $v$ is an element of every maximal nongenerating set of $G$.  Moreover, there exists $j$ such that $v\in V(G_j)$. 
By Proposition~\ref{prop:1cliqueSum}, $v$ is an element of every $N\in\mathcal{N}_c(G_j)$. Hence $v\in\Psi_j$, and so $v\in\bigcup_{i=1}^k\Psi_i$. 

Now, let $v\in \bigcup_{i} \Psi_i$, so that $v \in \Psi_j$ for some $j$. Then $v$ is an element of every $N\in\mathcal{N}_c(G_j)$, which implies that $v$ is an element of $N\cup\bigcup_{i\neq j}V(G_i)$ for every $N\in\mathcal{N}_c(G_j)$. Proposition~\ref{prop:1cliqueSum} implies that $v\in M$ for every $M\in\mathcal{N}(G)$, and so $v\in\Phi(G)$.
\end{proof}

\begin{example}
The 4-pan graph in Example~\ref{ex:4-pan} is the 1-clique sum of $C_4$ and $P_2$ with $c=v_2$. We have $\mathcal{N}_c(C_4)=\{\{v_2,v_3\},\{v_2,v_4\}\}$ and $\mathcal{N}_c(P_2)=\{\{v_2\}\}$. Hence 
\[
\mathcal{N}=\{\{v_2,v_3\}\cup\{v_1,v_2\},\{v_2,v_4\}\cup\{v_1,v_2\},\{v_2,v_3,v_4,v_5\}\cup\{v_2\}\},
\]
as we saw in Example~\ref{ex:4-pan}.
\end{example}

\begin{example}
The previous example generalizes to the 1-clique sum $G$ of $C_4$ and $P_n$ with the natural labelings $V(C_4)=\{c_1,\ldots, c_4\}$  and $V(P_n)=\{v_1,\ldots, v_n\}$ at $c_1=v_1$. We have 
\[
\mathcal{N}=\{\{c_2,v_1,\ldots,v_n\},\{c_4,v_1,\ldots,v_n\},\{c_2,c_3,c_4,v_1,\ldots,v_{n-1}\}\}.
\]
Hence the structure diagrams for even $n$ are the ones shown in Figure~\ref{fig:kite}. If $n$ is odd, then the parity of the structure classes are all reversed and the nim-values in the diagrams are swapped \cite[Proposition 9.2]{SiebenHypergraph}. Hence
\[
\nim(\DNG(G))=\begin{cases}
2, & \pty(n)=0\\
3, & \pty(n)=1,
\end{cases}
\]
and
\[
\nim(\GEN(G))=\begin{cases}
0, & \pty(n)=0\\
3, & \pty(n)=1.
\end{cases}
\]
\end{example}

%---------------%

\section{Graphs with a unique minimal generating set}

Graphs with a unique minimal generating set are common and relatively easy to analyze. The results of this section are closely related to the idea of a unique geodetic basis of~\cite{HaynesHenningTiller}.

\begin{prop}\label{prop:unique min gen set}
If $\mathcal{G}=\{L\}$, then $\mathcal{N}=\{\{l\}^c \mid l\in L\}$ and $\Phi=L^c$. 
\end{prop}

\begin{proof}
A set of vertices $S$ is a generating set if and only if $S$ is a superset of $L$. Hence the maximal nongenerating sets are the subsets of $V$ that miss exactly one element of $L$. The statement about the Frattini subset is an immediate consequence.
\end{proof}

The following result often makes it easy to recognize that a graph has a unique minimal generating set.

\begin{prop}
\label{prop:simplicialGenerate}
If $L$ is a generating set containing only simplicial vertices, then $\mathcal{G}=\{L\}$.
\end{prop}

\begin{proof}
Let $A \in \mathcal{G}$. Since $A$ is a generating set, $A$ contains all the simplicial vertices by Proposition~\ref{prop:simplicial}, so  $L\subseteq A$. Since $A$ is a minimal generating set and $L$ is a generating set, $A$ cannot be larger than $L$.  Thus $A=L$.
\end{proof}

\begin{example}
The graph shown in Figure~\ref{fig:barP3uK1uK1} has a generating set $L=\{v_3,v_4,v_5\}$ consisting of simplicial vertices so that $\mathcal{G}=\{L\}$.
\end{example}

\begin{figure}[h]
\begin{tikzpicture}[scale=1,auto]
\node (3) at (0,0) [ vert] {\scriptsize $v_3$};
\node (1) at (1,0) [ vert, fill=green] {\scriptsize $v_1$};
\node (2) at (1,1) [ vert, fill=green] {\scriptsize $v_2$};
\node (4) at (0,1) [ vert] {\scriptsize $v_4$};
\node (5) at (2,.5) [ vert] {\scriptsize $v_5$};
\path [b] (4) to (2) to (5) to (1) to (3) to (2) to (1) to (4) to (3);
\end{tikzpicture}

\caption{
\label{fig:barP3uK1uK1}
A graph with its Frattini subset highlighted in green. The simplicial vertices form a generating set.}
\end{figure}

\begin{prop}
\label{prop:uniqueDNG}
If $G$ is a graph such that $\mathcal{G}=\{L\}$, then $\nim(\DNG(G))=1-\pty(V)$.
\end{prop}

\begin{proof}
This follows from Proposition~\ref{prop:all-same} since $|N|=|V|-1$ for all $N\in\mathcal{N}$. 
\end{proof}

\begin{prop}
\label{prop:uniqueGEN}
If $G$ is a graph such that $\mathcal{G}=\{L\}$, then $\nim(\GEN(G))=\pty(V)$.
\end{prop}

\begin{proof}
By \cite[Proposition~7.2]{SiebenHypergraph}, 
\[
\nim(\GEN(G))=\begin{cases}
2, & |L|=1 \text{ and }\pty(V)=0\\
\pty(V), & \text{otherwise}.
\end{cases}
\]
A singleton set can only be a generating set if $G$ is the trivial graph but for the trivial graph $\pty(V)=1$. So the nim-value can never be $2$. 
\end{proof}

\begin{rem}
By using \cite[Propositions~7.6, 7.8]{SiebenHypergraph}, Propositions~\ref{prop:uniqueDNG} and~\ref{prop:uniqueGEN} can be generalized to graphs for which the family $\complement(\mathcal{N})$ of complements of the sets in $\mathcal{N}$ contains pairwise disjoint but not necessarily singleton sets. Figure~\ref{fig:disj} shows such a graph with 
\[
\complement(\mathcal{N})=\{\{v_1\},\{v_4\},\{v_5,v_6\},\{v_7,v_8\}\}.
\]
Recognizing these graphs does not seem easy, so we do not pursue this generalization further.
\end{rem}

\begin{figure}[h]
\begin{tikzpicture}[scale=1,auto]
\node (1) at (0,0) [ vert] {\scriptsize $v_1$};
\node (2) at (1,0) [ vert, fill=green] {\scriptsize $v_2$};
\node (3) at (2,0) [ vert, fill=green] {\scriptsize $v_3$};
\node (4) at (3,0) [ vert] {\scriptsize $v_4$};
\node (5) at (1,1) [ vert] {\scriptsize $v_5$};
\node (6) at (2,1) [ vert] {\scriptsize $v_6$};
\node (7) at (1,-1) [ vert] {\scriptsize $v_7$};
\node (8) at (2,-1) [ vert] {\scriptsize $v_8$};
\path [b] (1) to (2) to (3) to (4);
\path [b] (5) to (6) to (3) to (8) to (7) to (2) to (5);
\draw[rounded corners] (.5,.6) -- (.5,1.4) -- (2.5,1.4) -- (2.5,.6) -- cycle;
\draw[rounded corners] (.5,-.6) -- (.5,-1.4) -- (2.5,-1.4) -- (2.5,-.6) -- cycle;
\draw[rounded corners] (-.4,-.4) -- (-.4,.4) -- (0.4,.4) -- (.4,-.4) -- cycle;
\draw[rounded corners] (2.6,-.4) -- (2.6,.4) -- (3.4,.4) -- (3.4,-.4) -- cycle;
\end{tikzpicture}

\caption{
\label{fig:disj}
A graph with pairwise disjoint sets in $\complement(\mathcal{N})$ indicated by rectangles.}
\end{figure}

\subsection{Complete split graphs}
% https://math.stackexchange.com/questions/4014736/terminology-for-the-join-of-a-complete-graph-and-the-complement-of-a-complete-gr
Recall that the \emph{join} of two disjoint graphs is the graph in which we add each possible edge from one graph to the other. A \emph{complete split graph} is the join $K_m+\overline{K}_n$ of the complete graph $K_m$ and the complement graph $\overline{K}_n$.  The geodetic closure and convex hull of $S \subseteq V(K_m+\overline{K}_n)$ are the same since $[S]=S=I[S]$ if $|S \cap \overline{K}_n| \leq 1$, and $[S]=K_m \cup S=I[S]$ otherwise.  Thus, the results in this subsection apply to the games in~\cite{HaynesHenningTiller}. 

\begin{example}
The diamond graph shown in Figure~\ref{fig:diamond2} is the complete split graph $K_2+\overline{K}_2$.
\end{example}

A complete split graph has a unique minimal generating set.

\begin{prop}
If $G=K_m+\overline{K}_n$ and $n\ge 2$, then  $\mathcal{G}=\{V(\overline{K}_n)\}$.
\end{prop}
\begin{proof}
The vertices in $\overline{K}_n$ are simplicial in $G$. Since $n\ge 2$, $V(\overline{K}_n)$ is a generating set for $V(G)$. Hence the result follows from Proposition~\ref{prop:simplicialGenerate}.
\end{proof}

The following is an immediate consequence of Propositions~\ref{prop:uniqueDNG} and \ref{prop:uniqueGEN}.

\begin{corollary}
\label{cor:split}
If $G=K_m+\overline{K}_n$ and $n\ge 2$, then $\nim(\DNG(G))=1-\pty(m+n)$.
and
$\nim(\GEN(G))=\pty(m+n)$.
\end{corollary}

\subsection{Corona graphs} 
The \emph{corona} $H \circ K_1$ is formed from the graph $H$ by adding for each $v \in V(H)$ a new vertex $v'$ and a new edge $vv'$. We present results similar to \cite[Corollary 7, 15]{HaynesHenningTiller}.  

\begin{prop}
If $H$ is a nontrivial graph and $G=H \circ K_1$, then $\mathcal{G}=\{L\}$, where $L=\{v'\mid v\in V(H)\}$.
\end{prop}

\begin{proof}
The result follows from Proposition~\ref{prop:simplicialGenerate} since $L$ is a generating set containing simplicial vertices.
\end{proof}

Note that the previous proposition also holds for geodetic closures, hence the geodetic closure and the convex hull games are the same for corona graphs.

\begin{prop}
If $H$ is a nontrivial graph and $G=H \circ K_1$, then $\nim(\DNG(G))=1$.
\end{prop}

\begin{proof}
The result follows from Proposition~\ref{prop:all-same} since $|N|=2|V(H)|-1$ for all $N\in\mathcal{N}$.
\end{proof}

\begin{prop}
If $H$ is a nontrivial graph and $G=H \circ K_1$, then $\nim(\GEN(G))=0$.
\end{prop}

\begin{proof}
Proposition~\ref{prop:uniqueGEN} implies that $\nim(\GEN(G))=\pty(V(G))=0$ since $|V(G)|=2|V(H)|$.
\end{proof}

\begin{example}\label{example:coronacounterexample}
Consider the corona graph $G=K_{2,3}\circ K_1$. The geodetic closure and the convex hull are not the same operators in $G$ as can be seen using Figure~\ref{fig:K23}. But $L=\{v'\mid v\in V(K_{2,3})\}$ is the unique minimal generating set for both the geodetic closure and the convex hull operators. Hence the geodetic closure and the convex hull avoidance and achievement games are the same on $G$.
\end{example}

\subsection{Block graphs}
\ 
%\note{
%N: This appears in Haynes but called complete block graphs. See 
%{\tt https://en.wikipedia.org/wiki/Block\_graph} Perhaps reference a graph theory book like West.
%}

A \emph{block} of a graph is a maximal connected subgraph without a cut vertex. A \emph{block graph} is a graph whose blocks are complete graphs. The connected components of a block graph are block graphs. A connected block graph can be built from nontrivial complete graphs using $1$-clique sums. Because of this, a block graph is also referred to as a \emph{clique tree}.  The simplicial vertices form the unique minimal generating set for both the geodetic closure and the convex hull, so the results in this subsection apply to the games in~\cite{HaynesHenningTiller}.

\begin{example}
Figure~\ref{fig:block} shows a block graph with seven blocks and six cut vertices. The graph can be built from $K_4$, two copies of $K_3$, and four copies of $K_2$ using 1-clique sums. The set $L=\{v_1,v_5,v_6,v_7,v_{10},v_{12}\}$ of simplicial vertices is the unique minimal generating set. The set of cut vertices is $\{v_2,v_3,v_4,v_8,v_9,v_{11}\}$. Note that each cut vertex is contained on a geodesic between two simplicial vertices. For example, $v_9$ is contained on the geodesic between $v_1$ and $v_{12}$.
\end{example}

\begin{figure}[h]
\begin{tikzpicture}[scale=.8,auto]
\node (1) at (-.3,0) [vert] {\scriptsize $v_1$};
\node (2) at (1,0) [vert, fill=green] {\scriptsize $v_2$};
\node (3) at (-1,1) [vert, fill=green] {\scriptsize $v_3$};
\node (4) at (-1,-1) [vert, fill=green] {\scriptsize $v_4$};
\node (5) at (-2,-.4) [vert] {\scriptsize $v_5$};
\node (6) at (-2,-1.6) [vert] {\scriptsize $v_6$};
\node (7) at (-2,1) [vert] {\scriptsize $v_7$};
\node (8) at (2,.6) [vert, fill=green] {\scriptsize $v_8$};
\node (9) at (2,-.6) [vert, fill=green] {\scriptsize $v_9$};
\node (10) at (3,.6) [vert] {\scriptsize $v_{10}$};
\node (11) at (3,-.6) [vert, fill=green] {\scriptsize $v_{11}$};
\node (12) at (4,-.6) [vert] {\scriptsize $v_{12}$};
\path [b] (1) to (2) to (3) to (4) to (2);
\path [b] (3) to (1) to (4);
\path [b] (4) to (5) to (6) to (4);
\path [b] (3) to (7);
\path [b] (2) to (8) to (9) to (2);
\path [b] (8) to (10);
\path [b] (9) to (11);
\path [b] (11) to (12);
\end{tikzpicture}

\caption{
\label{fig:block}
A block graph with seven blocks and six cut vertices.}
\end{figure}

\begin{prop}
\label{prop:NforBlock}
If $G$ is a block graph, then the set $L$ of simplicial vertices is a unique minimal generating set.
\end{prop}

\begin{proof}
Every cut vertex is contained on some geodesic connecting two simplicial vertices. Hence the simplicial vertices form a generating set and the result follows from Proposition~\ref{prop:simplicialGenerate}.
\end{proof}

An alternate approach for the proof would be to use Proposition~\ref{prop:1cliqueSum} for the connected components followed by the use of Proposition~\ref{prop:maximal disjoint union}. Note that the previous proposition also holds for geodetic closures, hence the geodetic closure and the convex hull games are the same for block graphs.

\begin{corollary}
For block graphs, $\Phi$ consists of the cut vertices.
\end{corollary}

\begin{prop}
If $G$ is a block graph, then $\nim(\DNG(G))=1-\pty(V)$.
\end{prop}

\begin{proof}
The result follows from Propositions~\ref{prop:NforBlock} and \ref{prop:all-same}.
\end{proof}

\begin{prop}
If $G$ is a block graph, then $\nim(\GEN(G))=\pty(V)$.
\end{prop}

\begin{proof}
The result follows from Propositions~\ref{prop:NforBlock} and \ref{prop:uniqueGEN}.
\end{proof}

Block graphs include some basic graph families.

\begin{example}
The complete graph $K_n$ is a block graph with a single block. Every vertex is a simplicial vertex, so $\mathcal{G}=\{V\}$ and $\Phi=\emptyset$.
\end{example}

\begin{example}
A generalization of the windmill graphs is the block graph  $G=\Wd(\vec{n})$ in which $\vec{n} = (n_1, n_2, \ldots, n_\ell) \in  \mathbb{N}_{\ge 2}^\ell$ for $\ell \ge 2$ and we glue together complete graphs $K_{n_1}, \ldots, K_{n_\ell}$ at a common vertex $c$. In this case, $|V|=1-\ell + \sum_{i=1}^\ell n_i.$ The Frattini subset is $\Phi=\{c\}$. Note that when all entries in $\vec{n}$ are equal to $n$, we obtain the windmill graph $\Wd(n,\ell)$ in which $\ell$ copies of $K_n$ are glued together. 
\end{example}

\begin{example}\label{ex:trees are block graphs}
Forest graphs are block graphs. The simplicial vertices are the leaves. The Frattini subset consists of the vertices that are not leaves.  The path graph $P_n$ and the star graph $K_{1,n}$ are trees and hence block graphs.
\end{example}

%---------------%
\section{Games on graph families}
%---------------%

We study the impartial games on several graph families. We always assume that these games are played on a graph $G=(V,E)$ no matter what graph family we work with.

%---------------%
\subsection{Cycle graphs}\label{subsec:cycle graphs}
%---------------%

In this subsection we study the impartial games on the cycle graph $G=C_{n}$ with $n\ge 3$ that has vertex set $\{v_1,\ldots,v_n\}$ and $v_i$ is adjacent to $v_{i+1}$ if the indices are considered modulo $n$.  For cycle graphs, geodetic closure agrees with convex closure, so the results of this section also apply to the geodetic closure variations.

The following is straightforward to verify.

\begin{prop}\label{prop:maximal nongen for cycles}
If $\pty(n)=0$, then $\mathcal{N}=\{\{v_{i+1},\ldots, v_{i+n/2}\} \mid v_i\in V\}$. If $\pty(n)=1$, then $\mathcal{N}=\{\{v_{i+1},\ldots, v_{i+(n+1)/2}\} \mid v_i\in V\}$.
\end{prop}

We immediately get the following result.

\begin{corollary}
For cycle graphs, $\Phi=\emptyset$.
\end{corollary}

\begin{prop}
For cycle graphs, $\nim(\DNG(C_{n}))=\begin{cases}
1, & \text{if }n\equiv_4 1,2 \\
0, & \text{if }n\equiv_4 3,0.	
\end{cases}
$% $(\nim(\DNG(C_{n})))_{n\ge3}=(0,0,1,1,0,0,1,1,\ldots)$.
\end{prop}

\begin{proof}
In light of Proposition~\ref{prop:maximal nongen for cycles}, for $N\in\mathcal{N}$, we have
\[
\pty(N)=\begin{cases}
1, & \text{if }n\equiv_4 1,2 \\
0, & \text{if }n\equiv_4 3,0.
\end{cases}
\]
The result now follows from Proposition~\ref{prop:all-same}.
\end{proof}

\begin{prop}
For cycle graphs, $\nim(\GEN(C_{n}))=\pty(n)$.
\end{prop}

\begin{proof}
The proof of \cite[Theorem~1]{BuckleyHarary86} determines that the first player wins exactly when $n$ is odd. Since every first move is shrewd, the nim-value of every game on $C_{n}$ has nim-value $0$ or $1$ by Proposition~\ref{equivalent}.
\end{proof}

%---------------%
\subsection{Hypercube graphs}
%---------------%

For $n\geq 2$, we will denote the set of binary strings of length $n$ by $\{0,1\}^n$. That is,
\[
\{0,1\}^n := \{a_1a_2\cdots a_n \mid a_k \in \{0,1\} \}.
\]
The \emph{hypercube graph} of dimension $n\geq 0$, denoted by $Q_n$, is defined to be the graph whose vertices are elements of $\{0,1\}^n$ with two binary strings connected by an edge exactly when they differ by a single digit (i.e., the Hamming distance between the two vertices is equal to one). 

The following result is easily seen.

\begin{prop}\label{prop:hypercubeclosure} If $a_1a_2\cdots a_n, b_1b_2\cdots b_n\in V(Q_n)$, then
\[
[\{a_1a_2\cdots a_n, b_1b_2\cdots b_n\}]=\{c_1c_2\cdots c_n\mid c_i=a_i \text{ whenever } a_i=b_i\}.
\]
\end{prop}

The next example illustrates that the geodetic closure and convex hull games are not the same on $Q_n$ if $n\ge 3$.

\begin{example}\label{ex:cube}
Consider $S=\{000, 011, 110\}$, a subset of $Q_3$. Then $[S]=V$ but $I[S]=V \setminus \{101\}$.  That is, $S$ is a generating set for the convex hull but not for the geodetic closure. One can construct analogous examples in larger hypercubes.
\end{example}

We say that two binary strings $a_1a_2\cdots a_n$ and $b_1b_2\cdots b_n$ are \emph{antipodal} if $a_i\neq b_i$ for all $1\leq i\leq n$.

\begin{prop}
Every pair of antipodal vertices of $Q_n$ form a minimal generating set.
\end{prop}

\begin{proof}
Let $a, b \in V$ be antipodal. Since they differ in all $n$ entries, the interval between them consists of $2^n$ strings, which is all of the vertices. Since $I[a,b] = V$, we must also have $[\{a,b\}] = V$.
\end{proof}

For $1\leq i\leq n$ and $b\in\{0,1\}$, define 
\[
Q_{n,i,b}:=\{a_1a_2\cdots a_n \in V(Q_n)\mid a_i=b\}.
\]
That is, $Q_{n,i,b}$ is the collection of vertices whose $i$th entry is $b$.

\begin{prop}\label{prop:hypercubeN}
For hypercubes, $\mathcal{N}=\{Q_{n,i,b}\mid 1\leq i\leq n,b\in\{0,1\}\}$.
\end{prop}

\begin{proof}
We use induction on $n$. The $n=2$ case follows from Proposition~\ref{prop:maximal nongen for cycles} since $Q_2\cong C_4$. Now, suppose $V(P_2)=\{0,1\}$ so that $\mathcal{N}(P_2)=\{\{0\},\{1\}\}$ according to Propositions~\ref{prop:unique min gen set}, \ref{prop:NforBlock}, and Example~\ref{ex:trees are block graphs}. For $n\geq 3$, the map $f:V(Q_{n-1})\times V(P_2)\to V(Q_n)$ defined via
\[
f(a_1\cdots a_{n-1},b)=a_1\cdots a_{n-1}b
\]
is a graph isomorphism from $Q_{n-1}\boxprod P_2$ to $Q_n$.  Proposition~\ref{prop:maximal product} and induction implies that 
$\mathcal{N}(Q_{n-1}\boxprod P_2)=\mathcal{A}\cup\mathcal{B}$,
where 
\[
\begin{aligned}
\mathcal{A} &=\{Q_{n-1,i,b}\times V(P_2)\mid 1\leq i\leq n-1, b\in\{0,1\}\}, \\
\mathcal{B} &=\{V(Q_{n-1})\times \{b\}\mid b\in\{0,1\}\}.
\end{aligned}
\]
Hence
$
f(\mathcal{N}(Q_{n-1}\boxprod P_2))
=\{Q_{n,i,b}\mid 1\leq i\leq n,b\in\{0,1\}\}
$
since 
\[
\begin{aligned}
f(V(Q_{n-1,i,b}) \times V(P_2)) &= Q_{n,i,b}, \\
f(V(Q_{n-1}) \times \{b\}) &= Q_{n,n,b}
\end{aligned}
\]
for all $1\leq i\leq n-1$ and $b\in\{0,1\}$. 
\end{proof}

\begin{corollary}
For hypercubes, $\Phi=\emptyset$.
\end{corollary}

\begin{prop}
For hypercubes, $\nim(\DNG(Q_{n}))=0$.
\end{prop}

\begin{proof}
Every set in $\mathcal{N}$ has size $2^{n-1}$. By Proposition~\ref{prop:all-same}, $\nim(\DNG(Q_{n}))=0$.
\end{proof}
\begin{prop}
For hypercubes, $\nim(\GEN(Q_{n}))=0$.
\end{prop}

\begin{proof}
The second player wins by selecting the antipodal vertex to the choice of player one, and every antipodal pair forms a minimal generating set.
\end{proof}

%---------------------%
\subsection{Lattice graphs}
%---------------------%

We will refer to graphs of the form $P_{n_1}\boxprod \cdots \boxprod P_{n_d}$ as $d$-dimensional \emph{lattice graphs}, where each $n_i\geq 2$.  We assume the vertices are labeled in a natural way using $d$-tuples of the form $(a_{1},a_{2},\ldots,a_{d})$, where each $1\leq a_{i}\leq n_i$. Two vertices are connected exactly when a single component differs by 1. A 2-dimensional lattice graph $P_{m}\boxprod P_{n}$ with $2\le m\le n$ and $3\le n$ is called a \emph{grid graph}.

The geodetic closure and the convex hull operators are the same for grid graphs. However, they are not the same for lattice graphs in general, as seen in Example~\ref{ex:cube}.

For $1\leq i\leq d$ define
\[
N_{i}:=\{(a_1,\ldots,a_d)\mid a_i > 1 \}, \quad
N^{i}:=\{(a_1,\ldots,a_d)\mid a_i<n_i\}.
\]
%Note that $N_i$ and $N^i$ correspond to the maximal nongenerating sets of $P_{n_i}$.

\begin{prop}\label{prop:maximals lattice graphs}
For lattice graphs, $\mathcal{N}=\{N_{i}, N^i\mid 1\leq i\leq d\}$.
\end{prop}

\begin{proof}
This follows by induction together with Example~\ref{ex:trees are block graphs} and Propositions~\ref{prop:unique min gen set}  and~\ref{prop:maximal product}.
\end{proof}

We immediately get the following.

\begin{corollary}
For lattice graphs, $\Phi=\{(a_1,\ldots,a_d)\mid 1<a_i<n_i\}$.
\end{corollary}

For a grid graph, each maximal nongenerating set is the complement of the set of vertices lying along one side. So, for grid graphs there are exactly four maximal nongenerating sets and a total of 16 intersection sets when $m>2$ and 13 intersection sets when $m=2$.  Figure~\ref{fig:skeleton structure diagram for grid DNG} provides representative quotient structure digraphs for $\GEN(P_{m}\boxprod P_{n})$ and $\DNG(P_{m}\boxprod P_{n})$. The classes are indicated by visualizations of representative intersection subsets. We have labeled each node with the number of intersection sets being identified. In this case, the option relationship is easy to compute.

\begin{figure}[h]
\centering
\def\sc{.2}
\begin{tikzpicture}[scale=1.1,inner sep=1pt]
\node (0) at (0,4) {
\begin{tikzpicture}[scale=\sc]
\draw (0,0) -- (5,0) -- (5,2) -- (0,2) -- cycle;
\draw (0,1) -- (5,1);\draw (1,0) -- (1,2);\draw (4,0) -- (4,2);
\end{tikzpicture}
};
\node[label=left:\tiny $2\times$] (11) at (1,3) {%\tiny 2
\begin{tikzpicture}[scale=\sc]
\draw (0,0) -- (5,0) -- (5,2) -- (0,2) -- cycle;
\draw (0,1) -- (5,1);\draw (1,0) -- (1,2);\draw (4,0) -- (4,2);
\draw[fill=black!20] (1,0) -- (4,0) -- (4,1) -- (1,1) -- cycle;
\end{tikzpicture}
};
\node[label=left:\tiny $4\times$] (20) at (0,2) {%\tiny 4
\begin{tikzpicture}[scale=\sc]
\draw (0,0) -- (5,0) -- (5,2) -- (0,2) -- cycle;
\draw (0,1) -- (5,1);\draw (1,0) -- (1,2);\draw (4,0) -- (4,2);
\draw[fill=black!20] (1,0) -- (4,0) -- (4,1) -- (1,1) -- cycle;
\draw[fill=black!20] (0,0) -- (1,0) -- (1,1) -- (0,1) -- cycle;
\end{tikzpicture}
};
\node (21) at (2,2) {
\begin{tikzpicture}[scale=\sc]
\draw (0,0) -- (5,0) -- (5,2) -- (0,2) -- cycle;
\draw (0,1) -- (5,1);\draw (1,0) -- (1,2);\draw (4,0) -- (4,2);
\draw[fill=black!20] (1,0) -- (4,0) -- (4,1) -- (1,1) -- cycle;
\draw[fill=black!20] (1,1) -- (4,1) -- (4,2) -- (1,2) -- cycle;
\end{tikzpicture}
};
\node[label=left:\tiny $2\times$] (30) at (-1,1) {%\tiny 2
\begin{tikzpicture}[scale=\sc]
\draw (0,0) -- (5,0) -- (5,2) -- (0,2) -- cycle;
\draw (0,1) -- (5,1);\draw (1,0) -- (1,2);\draw (4,0) -- (4,2);
\draw[fill=black!20] (1,0) -- (4,0) -- (4,1) -- (1,1) -- cycle;
\draw[fill=black!20] (0,0) -- (1,0) -- (1,1) -- (0,1) -- cycle;
\draw[fill=black!20] (4,0) -- (5,0) -- (5,1) -- (4,1) -- cycle;
\end{tikzpicture}
};
\node[label=left:\tiny $2\times$] (31) at (1,1) {%\tiny 2
\begin{tikzpicture}[scale=\sc]
\draw (0,0) -- (5,0) -- (5,2) -- (0,2) -- cycle;
\draw (0,1) -- (5,1);\draw (1,0) -- (1,2);\draw (4,0) -- (4,2);
\draw[fill=black!20] (1,0) -- (4,0) -- (4,1) -- (1,1) -- cycle;
\draw[fill=black!20] (1,1) -- (4,1) -- (4,2) -- (1,2) -- cycle;
\draw[fill=black!20] (0,0) -- (1,0) -- (1,1) -- (0,1) -- cycle;
\draw[fill=black!20] (0,1) -- (1,1) -- (1,2) -- (0,2) -- cycle;
\end{tikzpicture}};
\node (40) at (0,0) {
\begin{tikzpicture}[scale=\sc]
\draw (0,0) -- (5,0) -- (5,2) -- (0,2) -- cycle;
\draw (0,1) -- (5,1);\draw (1,0) -- (1,2);\draw (4,0) -- (4,2);
\draw[fill=purple,opacity=0.3] (1,0) -- (4,0) -- (4,1) -- (1,1) -- cycle;
\draw[fill=purple,opacity=0.3] (1,1) -- (4,1) -- (4,2) -- (1,2) -- cycle;
\draw[fill=purple,opacity=0.3] (0,0) -- (1,0) -- (1,1) -- (0,1) -- cycle;
\draw[fill=purple,opacity=0.3] (0,1) -- (1,1) -- (1,2) -- (0,2) -- cycle;
\draw[fill=purple,opacity=0.3] (4,0) -- (5,0) -- (5,1) -- (4,1) -- cycle;
\draw[fill=purple,opacity=0.3] (4,1) -- (5,1) -- (5,2) -- (4,2) -- cycle;
\end{tikzpicture}
};
\draw[a] (0) -- (11);
\draw[a] (0) -- (20);
\draw[a] (11) -- (20);
\draw[a] (11) -- (21);
\draw[a] (20) -- (30);
\draw[a] (20) -- (31);
\draw[a] (21) -- (31);
\draw[a] (11) -- (31);
\draw[a,dashed] (30) -- (40);
\draw[a,dashed] (31) -- (40);
\draw[a,dashed] (20) -- (40);
\end{tikzpicture}
\hfil
\def\sc{.18}
\begin{tikzpicture}[scale=1.1,inner sep=1pt]
\node (0) at (0,4) {
\begin{tikzpicture}[scale=\sc]
\draw (0,0) -- (5,0) -- (5,4) -- (0,4) -- cycle;
\draw (0,1) -- (5,1);\draw (0,3) -- (5,3);\draw (1,0) -- (1,4);\draw (4,0) -- (4,4);
\draw[fill=black!20] (1,1) -- (4,1) -- (4,3) -- (1,3) -- cycle;
\end{tikzpicture}
};
\node[label=left:\tiny $2\times$] (11) at (-1,3) {
\begin{tikzpicture}[scale=\sc]
\draw (0,0) -- (5,0) -- (5,4) -- (0,4) -- cycle;
\draw (0,1) -- (5,1);\draw (0,3) -- (5,3);\draw (1,0) -- (1,4);\draw (4,0) -- (4,4);
\draw[fill=black!20] (1,1) -- (4,1) -- (4,3) -- (1,3) -- cycle;
\draw[fill=black!20] (0,1) -- (1,1) -- (1,3) -- (0,3) -- cycle;
\end{tikzpicture}
};
\node[label=left:\tiny $2\times$] (12) at (1,3) {
\begin{tikzpicture}[scale=\sc]
\draw (0,0) -- (5,0) -- (5,4) -- (0,4) -- cycle;
\draw (0,1) -- (5,1);\draw (0,3) -- (5,3);\draw (1,0) -- (1,4);\draw (4,0) -- (4,4);
\draw[fill=black!20] (1,1) -- (4,1) -- (4,3) -- (1,3) -- cycle;
\draw[fill=black!20] (1,0) -- (4,0) -- (4,1) -- (1,1) -- cycle;
\end{tikzpicture}
};
\node (21) at (-2,2) {
\begin{tikzpicture}[scale=\sc]
\draw (0,0) -- (5,0) -- (5,4) -- (0,4) -- cycle;
\draw (0,1) -- (5,1);\draw (0,3) -- (5,3);\draw (1,0) -- (1,4);\draw (4,0) -- (4,4);
\draw[fill=black!20] (1,1) -- (4,1) -- (4,3) -- (1,3) -- cycle;
\draw[fill=black!20] (0,1) -- (1,1) -- (1,3) -- (0,3) -- cycle;
\draw[fill=black!20] (4,1) -- (5,1) -- (5,3) -- (4,3) -- cycle;
\end{tikzpicture}
};
\node[label=left:\tiny $4\times$] (22) at (0,2) {
\begin{tikzpicture}[scale=\sc]
\draw (0,0) -- (5,0) -- (5,4) -- (0,4) -- cycle;
\draw (0,1) -- (5,1);\draw (0,3) -- (5,3);\draw (1,0) -- (1,4);\draw (4,0) -- (4,4);
\draw[fill=black!20] (1,1) -- (4,1) -- (4,3) -- (1,3) -- cycle;
\draw[fill=black!20] (0,1) -- (1,1) -- (1,3) -- (0,3) -- cycle;
\draw[fill=black!20] (0,0) -- (1,0) -- (1,1) -- (0,1) -- cycle;
\draw[fill=black!20] (1,0) -- (4,0) -- (4,1) -- (1,1) -- cycle;
\end{tikzpicture}
};
\node (23) at (2,2) {
\begin{tikzpicture}[scale=\sc]
\draw (0,0) -- (5,0) -- (5,4) -- (0,4) -- cycle;
\draw (0,1) -- (5,1);\draw (0,3) -- (5,3);\draw (1,0) -- (1,4);\draw (4,0) -- (4,4);
\draw[fill=black!20] (1,1) -- (4,1) -- (4,3) -- (1,3) -- cycle;
\draw[fill=black!20] (1,0) -- (4,0) -- (4,1) -- (1,1) -- cycle;
\draw[fill=black!20] (1,3) -- (4,3) -- (4,4) -- (1,4) -- cycle;
\end{tikzpicture}
};
\node[label=left:\tiny $2\times$] (31) at (-1,1) {
\begin{tikzpicture}[scale=\sc]
\draw (0,0) -- (5,0) -- (5,4) -- (0,4) -- cycle;
\draw (0,1) -- (5,1);\draw (0,3) -- (5,3);\draw (1,0) -- (1,4);\draw (4,0) -- (4,4);
\draw[fill=black!20] (1,1) -- (4,1) -- (4,3) -- (1,3) -- cycle;
\draw[fill=black!20] (1,0) -- (4,0) -- (4,1) -- (1,1) -- cycle;
\draw[fill=black!20] (0,0) -- (1,0) -- (1,1) -- (0,1) -- cycle;
\draw[fill=black!20] (0,1) -- (1,1) -- (1,3) -- (0,3) -- cycle;
\draw[fill=black!20] (4,0) -- (5,0) -- (5,1) -- (4,1) -- cycle;
\draw[fill=black!20] (4,1) -- (5,1) -- (5,3) -- (4,3) -- cycle;
\end{tikzpicture}
};
\node[label=left:\tiny $2\times$] (32) at (1,1) {
\begin{tikzpicture}[scale=\sc]
\draw (0,0) -- (5,0) -- (5,4) -- (0,4) -- cycle;
\draw (0,1) -- (5,1);\draw (0,3) -- (5,3);\draw (1,0) -- (1,4);\draw (4,0) -- (4,4);
\draw[fill=black!20] (1,1) -- (4,1) -- (4,3) -- (1,3) -- cycle;
\draw[fill=black!20] (1,0) -- (4,0) -- (4,1) -- (1,1) -- cycle;
\draw[fill=black!20] (0,0) -- (1,0) -- (1,1) -- (0,1) -- cycle;
\draw[fill=black!20] (0,1) -- (1,1) -- (1,3) -- (0,3) -- cycle;
\draw[fill=black!20] (0,3) -- (1,3) -- (1,4) -- (0,4) -- cycle;
\draw[fill=black!20] (1,3) -- (4,3) -- (4,4) -- (1,4) -- cycle;
\end{tikzpicture}
};
\node (41) at (0,0) {
\begin{tikzpicture}[scale=\sc]
\draw (0,0) -- (5,0) -- (5,4) -- (0,4) -- cycle;
\draw (0,1) -- (5,1);\draw (0,3) -- (5,3);\draw (1,0) -- (1,4);\draw (4,0) -- (4,4);
\draw[fill=purple,opacity=0.3] (1,1) -- (4,1) -- (4,3) -- (1,3) -- cycle;
\draw[fill=purple,opacity=0.3] (1,0) -- (4,0) -- (4,1) -- (1,1) -- cycle;
\draw[fill=purple,opacity=0.3] (0,0) -- (1,0) -- (1,1) -- (0,1) -- cycle;
\draw[fill=purple,opacity=0.3] (0,1) -- (1,1) -- (1,3) -- (0,3) -- cycle;
\draw[fill=purple,opacity=0.3] (0,3) -- (1,3) -- (1,4) -- (0,4) -- cycle;
\draw[fill=purple,opacity=0.3] (1,3) -- (4,3) -- (4,4) -- (1,4) -- cycle;
\draw[fill=purple,opacity=0.3] (4,1) -- (5,1) -- (5,3) -- (4,3) -- cycle;
\draw[fill=purple,opacity=0.3] (4,0) -- (5,0) -- (5,1) -- (4,1) -- cycle;
\draw[fill=purple,opacity=0.3] (4,3) -- (5,3) -- (5,4) -- (4,4) -- cycle;
\end{tikzpicture}
};
\draw[a] (0) -- (11);
\draw[a] (0) -- (12);
\draw[a] (0) -- (22);

\draw[a] (11) -- (21);
\draw[a] (11) -- (22);
\draw[a] (11) -- (31);

\draw[a] (12) -- (22);
\draw[a] (12) -- (23);
\draw[a] (12) -- (32);

\draw[a] (21) -- (31);

\draw[a] (22) -- (31);
\draw[a] (22) -- (32);

\draw[a] (23) -- (32);

\draw[a,dashed] (31) -- (41);
\draw[a,dashed] (32) -- (41);
\draw[a,dashed] (22) -- (41);
\end{tikzpicture}

\caption{
\label{fig:skeleton structure diagram for grid DNG}
Schematic structure digraphs for $\DNG(P_{m} \boxprod P_{n})$ and $\GEN(P_{m}\boxprod P_{n})$ with $m=2$ and $m>2,$ respectively. The shaded regions contain the selected vertices. The dashed arrows and the bottom-most completely shaded grid are only present in $\GEN(P_{m}\boxprod P_{n})$.}
\end{figure}

\begin{prop}\label{prop:Lattice DNG}
For grid graphs,
\[
\nim(\DNG(P_{m}\boxprod P_{n}))=\begin{cases}
0,&  m=2 \\
2\pty(m+n), & m\ge 3.
\end{cases}
\]
\end{prop}

\begin{proof}
If $\pty(n)=\pty(m)$, the result follows from Propositions~\ref{prop:all-same} and~\ref{prop:maximals lattice graphs}. If $\pty(m)\neq\pty(n)$, then one can determine the parity of the structure classes and use type calculus to verify that the simplified structure diagrams for the cases $m=2$ and $m>2$ are as shown in Figure~\ref{fig:Grid Odd Even DNG}.
\end{proof}

\begin{figure}[h]
\centering
\begin{tabular}{ccc}
\includegraphics[scale=.4]{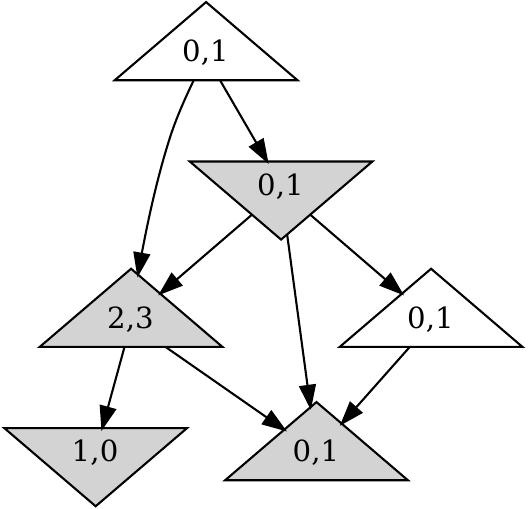} & &
\includegraphics[scale=.4]{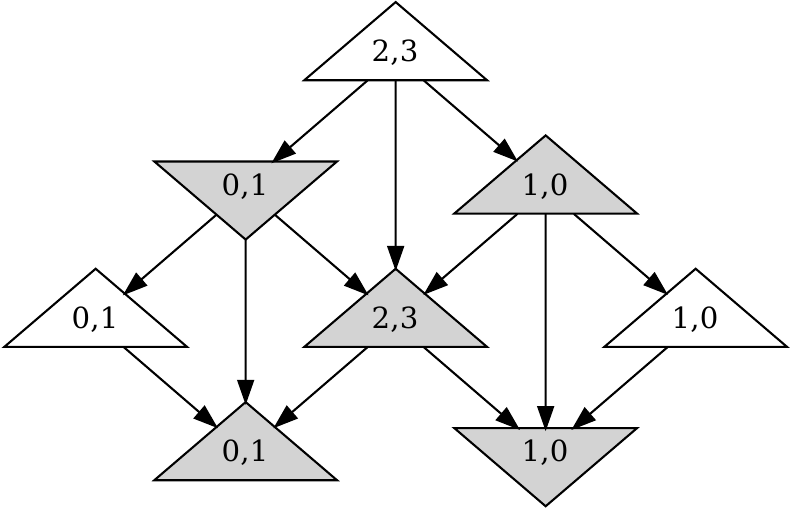}\\
$m=2$ & & $m>2$
\end{tabular}

\caption{
\label{fig:Grid Odd Even DNG}
Simplified structure diagrams for $\DNG(P_{m}\boxprod P_{n})$ with $\pty(m)$ different from $\pty(n)$.}
\end{figure}

\begin{prop}\label{prop:Lattice GEN}
For grid graphs, $\nim(\GEN(P_{m}\boxprod P_{n}))=\pty(mn)$.
\end{prop}

\begin{proof}
One can determine the parity of the structure classes and use type calculus to verify that the simplified structure diagrams for the cases $m=2$ and $m>2$ are as shown in Figures~\ref{fig:Grid2GEN} and~\ref{fig:GridGEN}.
\end{proof}

\begin{figure}[h]
\centering
\begin{tabular}{cc}
\quad \quad \includegraphics[scale=.4]{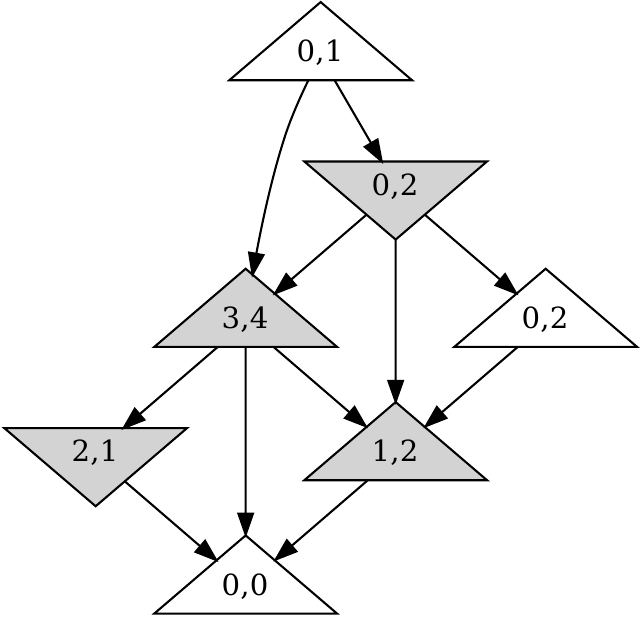} \quad \quad & \quad \quad 
\includegraphics[scale=.4]{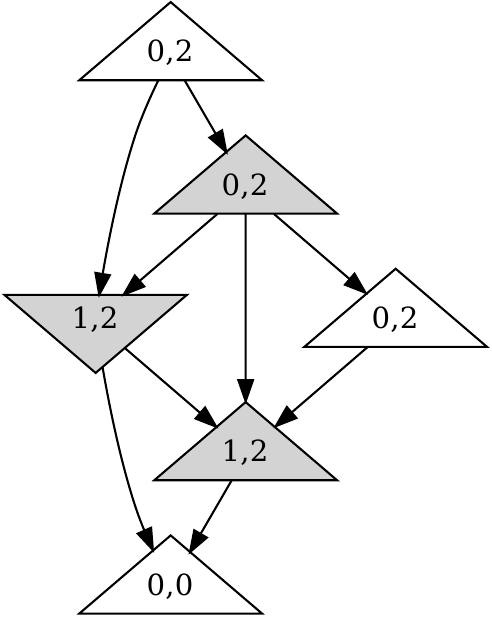} \quad \quad\\
2 $\times$ odd & 2 $\times$ even 
\end{tabular}
\caption{
\label{fig:Grid2GEN}
Simplified structure diagrams for $\GEN(P_{2}\boxprod P_{n})$.}
\end{figure}

\begin{figure}[h]
\centering
\begin{tabular}{ccc}
\quad \quad \includegraphics[scale=.4]{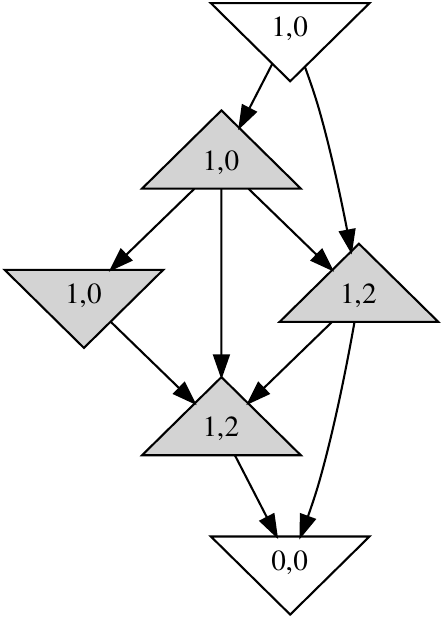} \quad \quad & \quad \quad \includegraphics[scale=.4]{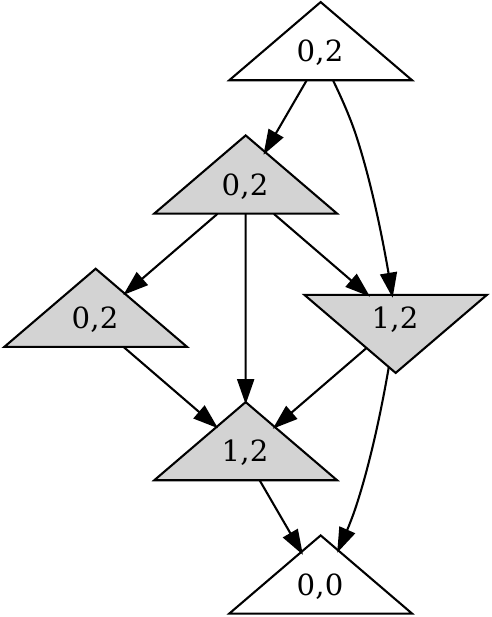} \quad \quad & \quad\quad \includegraphics[scale=.4]{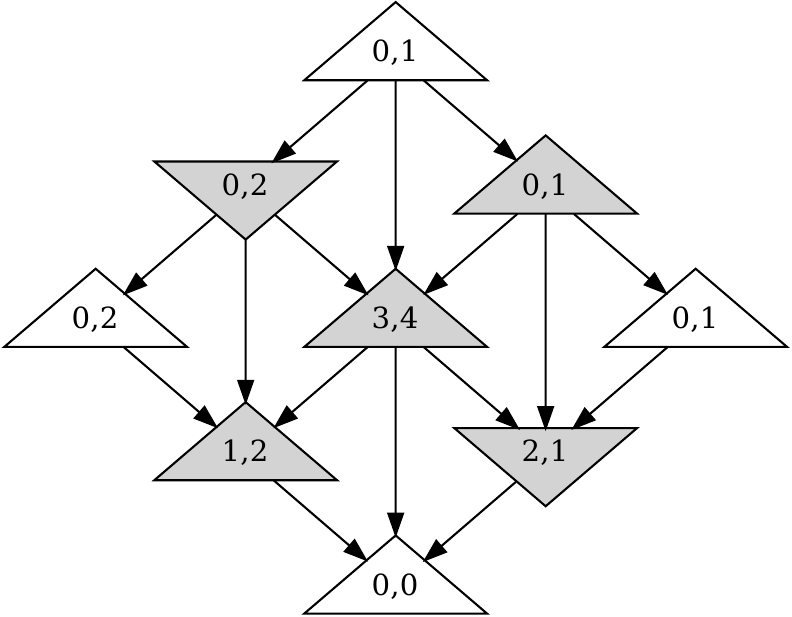} \quad\quad\\
odd $\times$ odd & even $\times$ even & even $\times$ odd
\end{tabular}
\caption{
\label{fig:GridGEN}
Simplified structure diagrams for $\GEN(P_{m}\boxprod P_{n})$ with $2<m$.}
\end{figure}

%---------------------%
\subsection{Complete multipartite graphs} 
%---------------------%
In this section we consider the \emph{complete multipartite graph} $G=K_{m_1,\ldots,m_k}$ with $k\ge 2$, $m_1\le m_2\le \cdots \le m_k$. The parts of $G$ that contain only one vertex are called \emph{small}, while the parts containing at least two vertices are called \emph{large}. We let $\sigma:=|\{i\mid m_i=1\}|$ and $\lambda:=|\{i\mid m_i\ge 2\}|$ be the number of small and large parts, respectively. Note that $k=\sigma+\lambda$.

If $\lambda=0$, then $G$ is a complete graph. If $\lambda=1$, then $G$ is a complete split graph where the small parts form a clique and the large part is an independent set. 
So we have the following result.

\begin{prop}
For a complete multipartite graph $G$ with $\lambda\in\{0,1\}$,  $\nim(\DNG(G))=1-\pty(V)$ and $\nim(\GEN(G))=\pty(V)$.
\end{prop}

\begin{prop}\label{max-multi}
If $G$ is a complete multipartite graph with $\lambda\ge 2$, then $\mathcal{N}$ consists of sets that contain exactly one vertex from each part of $G$.
\end{prop}

\begin{proof}
Let $N=\{v_1,v_2,\ldots,v_k\}$ such that each vertex $v_i$ is from a different part. Any geodesic between two vertices in $N$ has length $1$, so $[N]=N\not=V$.

Now consider $N':=N \cup \{v\}$ for some vertex $v \in V \setminus N$. Then there exists an $i \in \{1,2,\ldots,k\}$ such that $\{v,v_i\}$ are in the same large part, which we denote $C_1$. Since the geodesic between $v$ and $v_i$ is size 2, all vertices in $V \setminus C_1$ lie in these geodesics and are included in $[N']$, which shows $V \setminus C_1 \subset [N']$. Since $\lambda \ge 2,$ there exists an additional large part, which we denote $C_2$, such that $C_2 \subset [N']$. The remaining vertices in $C_1$ lie in geodesics of size 2 between the vertices in $C_2$, which shows $C_1 \subset [N']$ and thus, $[N']=V$. 

Now let $L$ be a maximal nongenerating set. Since two vertices from a single large part generate $G$, $L$ must contain at most one vertex from each part. Since $L$ is maximal, $L$ must contain exactly one vertex from each part.
\end{proof}

\begin{corollary}
\label{cor:multipartite frattini}
For a complete multipartite graph, $\Phi$ consists of all the vertices from the small parts. 
% for $k\ge 2$ and $\lambda\ge 1$ ?
\end{corollary}

\begin{prop}
For a complete multipartite graph $G$ with $\lambda\ge 2$, $\nim(\DNG(G))=\pty(\lambda+\sigma)$.
\end{prop}

\begin{proof}
The size of each maximal nongenerating set is $k$ by Proposition~\ref{max-multi}, so the result follows by Proposition~\ref{prop:all-same}.
\end{proof}

\begin{prop}
For a complete multipartite graph $G$ with $\lambda\ge 2$,
\[
\nim(\GEN(G))=\begin{cases}
%\pty(V), & \lambda\in\{0,1\} \\
2\pty(\sigma), & \pty(\lambda)=0 \\
\pty(\sigma), & \pty(\lambda)=1.
\end{cases}
\]
\end{prop}

\begin{proof}
For a nonterminal structure class $X_I$, define $\iota(I):=|I\cap(V\setminus \Phi)|$ to be the number of vertices from large parts contained in $I$. Note that $\iota(\Phi)=0$ and $\iota(N)=\lambda$ for all $N\in\mathcal{N}$.
Proposition~\ref{max-multi} and Corollary~\ref{cor:multipartite frattini} imply that $X_I$ is a nonterminal structure class if and only if $I$ contains every vertex in a small part and at most one vertex from each large part of $G$. It also implies that if $X_J$ is an option of $X_I$, then $\iota(J)=\iota(I)+1$ or $J=V$.

If $\iota(I)=\lambda$, then the only option of $X_I$ is $X_V$. Suppose $1\leq \iota(I)<\lambda$. Then $I$ contains a vertex from at least one large part $C$, so $X_I$ has $X_V$ as an option by selecting a second element of $C$.  Every other option $X_J$ satisfies $\iota(J)=i(I)+1$, since $J$ will intersect exactly one more large part than $I$ does. Type calculus and induction on $\iota(I)$ from $\lambda$ to $1$ shows that  
\[
\type(X_I)=\begin{cases}
(1-\pty(\lambda-\iota(I)),2,1), & \pty(\sigma+\lambda)=1\\
(\pty(\lambda-\iota(I)),1,2), & \pty(\sigma+\lambda)=0
\end{cases}
\]
for $1\le \iota(I)\le\lambda$, as shown in Figure~\ref{fig:multipartiteLargeLambda}.

By Corollary~\ref{cor:multipartite frattini}, every option of $X_{\Phi}$ is of the form $X_J$ with $J=\lceil \Phi\cup \{v\}\rceil=\Phi\cup \{v\}$ for some $v\in V\setminus \Phi$. Hence every option $X_J$ of $X_{\Phi}$ satisfies $\iota(J)=1$. Thus 
\[
\type(X_\Phi)= 
\begin{cases}
(0,0,1), & \pty(\sigma)=0,\pty(\lambda)=1\\
(1,2,0), & \pty(\sigma)=1,\pty(\lambda)=0\\
(1,1,0), & \pty(\sigma)=1,\pty(\lambda)=1\\
(0,0,2), & \pty(\sigma)=0,\pty(\lambda)=0,
\end{cases}
\]
which justifies the claim.
\end{proof}

\begin{figure}[h]
\centering
\begin{tabular}{ccccc}
\begin{tikzpicture}[scale=.9]
\node at (0,-.25) {};
\node at (0,1) {$\lambda$};
\node at (0,2) {$\lambda-1$};
\node at (0,3) {$\lambda-2$};
\node at (0,4) {$\vdots$};
\end{tikzpicture} &
\includegraphics[scale=.4]{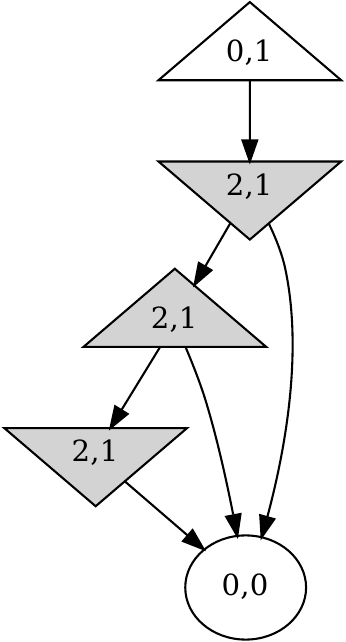} &
\includegraphics[scale=.4]{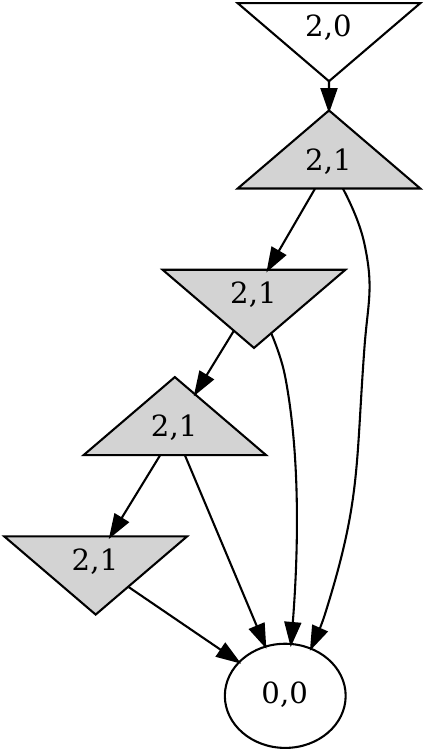} &
\includegraphics[scale=.4]{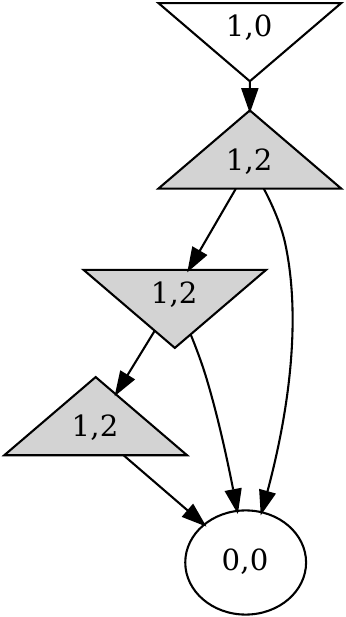} &
\includegraphics[scale=.4]{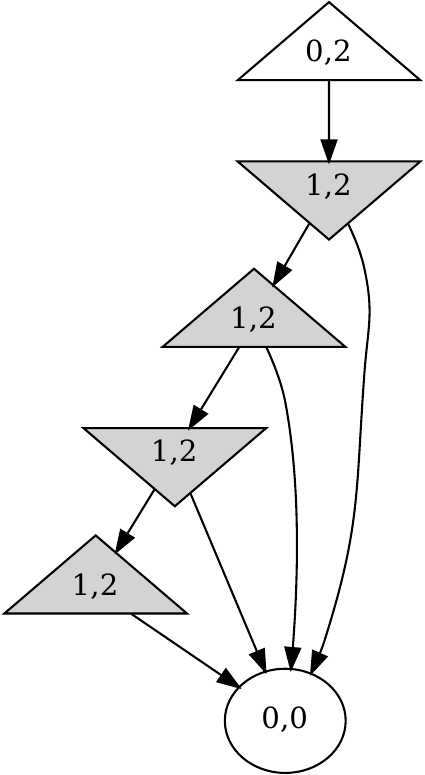}\\
$\iota(I)$ & $\pty(\sigma)=0$ & $\pty(\sigma)=1$ & $\pty(\sigma)=1$ & $\pty(\sigma)=0$\\
        & $\pty(\lambda)=1$ & $\pty(\lambda)=0$ & $\pty(\lambda)=1$ & $\pty(\lambda)=0$
\end{tabular}

\caption{\label{fig:multipartiteLargeLambda}
Simplified structure diagrams for $\GEN(G)$ on a complete multipartite graph $G$ with $\lambda\ge 2$. The number of gray triangles is $\lambda$ so that $\tau(\Phi)=\lambda$.  The parity of $X_V$ can be both 0 and 1. This is indicated by an oval.}
\end{figure}

The following special case handles complete bipartite graphs.

\begin{corollary}
If $n\ge m \ge 2$, then $\nim(\DNG(K_{m,n}))=0=\nim(\GEN(K_{m,n}))$.
\end{corollary}

%---------------------%
\subsection{Wheel graphs}
%---------------------%

For $n \geq 5$, the wheel graph $W_{n}$ is the join $K_1+C_{n-1}$ with vertices in $V(K_1)=\{c\}$ and $V(C_{n-1})=\{v_1, \ldots, v_{n-1}\}$.  It is easy to see that the geodetic closure and the convex hull of a set are the same for wheel graphs, so the results of this section also apply to geodetic closure games.

\begin{prop}\label{cor:max-wheel}
For wheel graphs, $\mathcal{N}$ consists of complements of sets containing two neighboring noncentral vertices.    
\end{prop}

\begin{proof}
It is clear that  $\{v_i,v_{i+1}\}^c$ is a maximal nongenerating set for all $i$.  To see the other containment, let $N \in \mathcal{N}$.     Note that $V \setminus \{c\}$ is a generating set, so  $v_i \not\in N$ for some $i$. Since $v_i$ is on a geodesic between $v_{i-1}$ and $v_{i+1}$, we may assume without loss of generality that $v_{i+1} \not\in N$.  Hence $N \subseteq \{v_i,v_{i+1}\}^c$.  Since $\{v_i,v_{i+1}\}^c$ is a  maximal nongenerating set, we conclude that $N=\{v_i,v_{i+1}\}^c$.
\end{proof}

\begin{corollary}
For wheel graphs, $\Phi = \{c\}$. 
\end{corollary}

\begin{prop}
For wheel graphs, $\nim(\DNG(W_{n}))=\pty(n)$.
\end{prop}

\begin{proof}
By Proposition~\ref{cor:max-wheel}, each set in $\mathcal{N}$ has size $n-2$. The result now follows by Proposition~\ref{prop:all-same}.
\end{proof}

\begin{figure}[h]
\begin{tikzpicture}[scale=.6]
\node[vert] (r4) at (2,0) {\scriptsize$\alpha_4$};
\node[vert] (r3) at (2.4,1.4) {\scriptsize$\alpha_3$};
\node[vert] (r1) at (5.6,1.4) {\scriptsize$\alpha_2$};
\node[vert] (r0) at (6.0,0) {\scriptsize$\alpha_1$};
\path [b] (r0) -- (r1);
\path [b] (r4) -- (r3);
\node[outer sep=2mm] (rr3) at (r3) {};
\node[outer sep=2mm] (rr0) at (r0) {};
\node[outer sep=2mm] (rr4) at (r4) {};
\node[outer sep=2mm] (rr1) at (r1) {};
\draw[<->,red] (rr1) edge (rr4);
\draw[<->,red] (rr0) edge (rr3);
\end{tikzpicture}
$\quad$ 
\begin{tikzpicture}[rotate=24,scale=.8]
\node[vert] (l0) at (-2,0) {\scriptsize$\alpha_1$};
\node[vert] (l1) at (-2.6,1.4) {\scriptsize$\alpha_2$};
\node[vert] (l2) at (-4,2) {\scriptsize$\alpha_3$};
\node[vert] (l3) at (-5.4,1.4) {\scriptsize$\alpha_4$};
\path [b] (l0) to (l1) to (l2) to (l3);
\node[outer sep=2mm] (ll2) at (l2) {};
\node[outer sep=2mm] (ll0) at (l0) {};
\node[outer sep=2mm] (ll3) at (l3) {};
\node[outer sep=2mm] (ll1) at (l1) {};
\draw[<->,red] (ll2) edge (ll0); 
\draw[<->,red] (ll3) edge (ll1);
\end{tikzpicture}

\caption{\label{fig:UE}Pairing strategy represented by arrows for the usual endgame. Two vertices are adjacent if and only if they are connected by an edge.}
\end{figure}

\begin{prop}
For wheel graphs, $\nim(\GEN(W_n))=\begin{cases}
2, & n=5\\
\pty(n), & n\geq 6.
\end{cases}$
\end{prop}

\begin{proof}
The $n=5$ case is handled by Example~\ref{ex:W5}.

For $n\ge 6$, a common position in the game is when only four unchosen vertices remain in two pairs of neighboring noncentral vertices, as shown in Figure~\ref{fig:UE}. The second player can win using the pairing strategy shown with red arrows in the figure. We will refer to this as the \emph{usual endgame}.

Suppose that $n$ is even and $n\ge 6$. We show that the second player has a winning strategy. The second player can ensure that $c$ is chosen in the first two moves. The second player can partition the remaining $n-2$ vertices into adjacent pairs. 
The second player always selects the pair of the vertex chosen by the first player until four vertices---two pairs---remain. The second player now wins by the usual endgame.

Now suppose that $n$ is odd and $n\ge 6$. We will show that the second player has a winning strategy for $\GEN(W_n) + *1$. Let $s$ denote the only available move in $*1$.  Since $n$ is odd, $n-1$ is even and the second player can partition the $n-1$ vertices on the rim into adjacent pairs.  The second player will pair $c$ and $s$ to complete the pairing. The second player will choose the pair of the vertex that was immediately played by the first player until there are eight elements left, where an element is either a vertex or $s$.  Let $S$ be the set of remaining eight unselected elements.  We now do a case analysis for the remainder of the game.

{\sl Case I:}   The elements $c$ and $s$ are not in $S$.  Then the second player continues to take the pair of the previously played vertex until there are four unchosen vertices remaining, and then the second player wins by the usual endgame.

{\sl Case II:}  The elements $c$ and $s$ are in $S$. 

{\sl Case IIA:}  The first player chooses either $c$ or $s$ as the first choice from $S$.   Then the second player chooses whichever of $c$ or $s$ remains, continues to take the pair of the previously played vertex until there are four unchosen vertices remaining, and then wins by the usual endgame.

{\sl Case IIB:} If the first player chooses any vertex $z$ other than $c$ or $s$, then the second player will select $s$.   The remainder of the game depends on the adjacencies of the rim vertices, and we consider four cases.  These four cases listed below are shown in Figure~\ref{fig:WheelGENCases2Bret}. In the figure, when the first player selects a vertex at the tail of an arrow, the second player responds by selecting the vertex at the head of the same arrow. If the first player selects a white vertex, then the game ends after the reply of the second player. If the first player selects a gray vertex, then the game ends through the usual endgame.  

{\sl Case IIB1:}   The five rim vertices $\{\alpha_1, \alpha_2, \alpha_3, \alpha_4, \alpha_5\}$ form a path as in Figure~\ref{fig:WheelGENCases2Bret}(1), where $\alpha_1$ and $\alpha_5$ are not adjacent because $n>6$.  

\begin{itemize}
\item If the first player picks $\alpha_2$ or $\alpha_4$, then the second player picks the other to win. 

\item If the first player picks $\alpha_3$ or $c$, the second player picks the other to win by the usual endgame.

\item If the first player picks $\alpha_1$ or $\alpha_5$, the second player picks $c$ to win by the usual endgame.
\end{itemize}

We explain the strategy for the remaining three cases only via the diagram.

{\sl Case IIB2:} Four of the rim vertices $\{\alpha_1,\alpha_2,\alpha_3,\alpha_4\}$ form a path as in Figure~\ref{fig:WheelGENCases2Bret}(2), but the fifth rim vertex $\beta$ is not adjacent to any $\alpha_i$.  

{\sl Case IIB3:} Three of the rim vertices $\{\alpha_1,\alpha_2,\alpha_3\}$ form a path as in Figure~\ref{fig:WheelGENCases2Bret}(3), and the other two rim vertices $\{\beta_1,\beta_2\}$ form a second disjoint path.  

{\sl Case IIB4:} The five rim vertices form two disjoint paths $\{\alpha_1,\alpha_2\}$ and $\{\beta_1,\beta_2\}$ and a singleton $\{\gamma\}$, as in Figure~\ref{fig:WheelGENCases2Bret}(4).  
\end{proof}

\begin{figure}[h]
\hfil
\begin{tikzpicture}[scale=.65,inner sep=1mm]
\node[vert] (r4) at (2,0) {$\scriptstyle \alpha_5$};
\node[vert, fill=white] (r3) at (2.4,1.4) {$\scriptstyle \alpha_4$};
\node[vert] (r2) at (4,2) {$\scriptstyle \alpha_3$};
\node[vert, fill=white] (r1) at (5.6,1.4) {$\scriptstyle \alpha_2$};
\node[vert] (r0) at (6.0,0) {$\scriptstyle \alpha_1$};
\node[vert] (cr) at (4,0) {$\scriptstyle c$};
\node[circle,outer sep=2mm] (rr4) at (r4) {};
\node[circle,outer sep=2mm] (ccr) at (cr) {};
\node[circle,outer sep=2mm] (rr3) at (r3) {};
\node[circle,outer sep=1.7mm] (rr2) at (r2) {};
\node[circle,outer sep=2mm] (rr1) at (r1) {};
\node[circle,outer sep=2mm] (rr0) at (r0) {};
\path [b] (cr) -- (r0);
\path [b] (cr) -- (r1);
\path [b] (cr) -- (r2);
\path [b] (cr) -- (r3);
\path [b] (cr) -- (r4);
\path [b] (r0) -- (r1);
\path [b] (r2) -- (r1);
\path [b] (r2) -- (r3);
\path [b] (r4) -- (r3);
\draw[<->,red] (rr2) edge [bend right] (ccr);
\draw[<->,red] (rr1) edge (rr3);
\draw[->,red] (rr0) edge [bend left] (ccr);
\draw[->,red] (rr4) edge [bend right] (ccr);
\node at (4,-1.3) {(1)};
\end{tikzpicture}
\hfil
\begin{tikzpicture}[scale=.65,inner sep=1mm]
\node[vert] (r4) at (2,0) {$\scriptstyle \beta$};
\node[vert, fill=white] (r3) at (2.4,1.4) {$\scriptstyle \alpha_4$};
\node[vert, fill=white] (r2) at (4,2) {$\scriptstyle \alpha_3$};
\node[vert, fill=white] (r1) at (5.6,1.4) {$\scriptstyle \alpha_2$};
\node[vert, fill=white] (r0) at (6.0,0) {$\scriptstyle \alpha_1$};
\node[vert] (cr) at (4,0) {$\scriptstyle c$};
\node[circle,outer sep=2mm] (rr4) at (r4) {};
\node[circle,outer sep=2mm] (ccr) at (cr) {};
\node[circle,outer sep=2mm] (rr3) at (r3) {};
\node[circle,outer sep=2mm] (rr2) at (r2) {};
\node[circle,outer sep=2mm] (rr1) at (r1) {};
\node[circle,outer sep=2mm] (rr0) at (r0) {};
\path [b] (cr) -- (r0);
\path [b] (cr) -- (r1);
\path [b] (cr) -- (r2);
\path [b] (cr) -- (r3);
\path [b] (cr) -- (r4);
\path [b] (r0) -- (r1);
\path [b] (r2) -- (r1);
\path [b] (r2) -- (r3);
\draw[<->,red] (rr1) edge (rr3);
\draw[<->,red] (rr0) edge (rr2);
\draw[<->,red] (rr4) edge [bend right] (ccr);
\node at (4,-1.3) {(2)};
\end{tikzpicture}
\hfil
\begin{tikzpicture}[scale=.65,inner sep=1mm]
\node[vert, fill=white] (r4) at (2,0) {$\scriptstyle \beta_2$};
\node[vert, fill=white] (r3) at (2.4,1.4) {$\scriptstyle \beta_1$};
\node[vert] (r2) at (4,2) {$\scriptstyle \alpha_3$};
\node[vert, fill=white] (r1) at (5.6,1.4) {$\scriptstyle \alpha_2$};
\node[vert] (r0) at (6.0,0) {$\scriptstyle \alpha_1$};
\node[vert] (cr) at (4,0) {$\scriptstyle c$};
\node[circle,outer sep=2mm] (rr4) at (r4) {};
\node[circle,outer sep=2mm] (ccr) at (cr) {};
\node[circle,outer sep=2mm] (rr3) at (r3) {};
\node[circle,outer sep=1.8mm] (rr2) at (r2) {};
\node[circle,outer sep=2mm] (rr1) at (r1) {};
\node[circle,outer sep=2mm] (rr0) at (r0) {};
\path [b] (cr) -- (r0);
\path [b] (cr) -- (r1);
\path [b] (cr) -- (r2);
\path [b] (cr) -- (r3);
\path [b] (cr) -- (r4);
\path [b] (r0) -- (r1);
\path [b] (r2) -- (r1);
\path [b] (r4) -- (r3);
\draw[->,red] (rr2) edge [bend right] (ccr);
\draw[<->,red] (rr1) edge (rr3);
\draw[<->,red] (rr0) edge [bend left] (ccr);
\draw[->,red] (rr4) edge (rr1);
\node at (4,-1.3) {(3)};
\end{tikzpicture}
\hfil
\begin{tikzpicture}[scale=.65,inner sep=1mm]
\node[vert] (r4) at (180:2) {$\scriptstyle \gamma$};
\node[vert, fill=white] (r3) at (135:2) {$\scriptstyle \beta_2$};
\node[vert, fill=white] (r2) at (90:2) {$\scriptstyle \beta_1$};
\node[vert, fill=white] (r1) at (45:2) {$\scriptstyle \alpha_2$};
\node[vert, fill=white] (r0) at (0:2) {$\scriptstyle \alpha_1$};
\node[vert] (cr) at (0,0) {$\scriptstyle c$};
\node[circle,outer sep=2mm] (rr4) at (r4) {};
\node[circle,outer sep=2mm] (ccr) at (cr) {};
\node[circle,outer sep=2mm] (rr3) at (r3) {};
\node[circle,outer sep=2mm] (rr2) at (r2) {};
\node[circle,outer sep=2mm] (rr1) at (r1) {};
\node[circle,outer sep=2mm] (rr0) at (r0) {};
\path [b] (cr) -- (r0);
\path [b] (cr) -- (r1);
\path [b] (cr) -- (r2);
\path [b] (cr) -- (r3);
\path [b] (cr) -- (r4);
\path [b] (r0) -- (r1);
\path [b] (r2) -- (r3);
\draw[<->,red] (rr1) edge [bend right] (rr2);
\draw[<->,red] (rr0) edge (rr3);
\draw[<->,red] (rr4) edge [bend right] (ccr);
\node at (0,-1.3) {(4)};
\end{tikzpicture}
\hfil\phantom{}

\caption{\label{fig:WheelGENCases2Bret}Strategies in Case IIB.  Two vertices are adjacent if and only if they are connected by an edge.}
\end{figure}

\subsection{Generalized wheel graphs}
The generalized wheel graph $W_{m,n}$ with $m\ge 2$ and $n\ge 3$ is the join $\overline{K}_m+C_n$ with $m+n$ total vertices in $V(\overline{K}_m)=\{c_1, \ldots, c_m\}$ and $V(C_n)=\{v_1, \ldots, v_n\}$.
The geodetic closure and the convex hull operators are the same for $W_{m,3}$ and $W_{2,4}$, so the geodetic and convex games are the same. Otherwise, the geodetic closure and convex hull games are different because the minimal generating sets are different.

Since $W_{m,3}$ is a complete split graph, we immediately have the following consequence of Corollary~\ref{cor:split}.

\begin{prop}
For the generalized wheel graphs $W_{m,3}$, $\nim(\DNG(W_{m,3}))=\pty(m)$ and $\nim(\GEN(W_{m,3}))=1-\pty(m)$.
\end{prop}

\begin{prop}\label{prop:max-genwheel}
For $W_{m,n}$ with $n\ge 4$, 
\[
 \mathcal{N} = \{ \{c_i, v_j, v_{j+1}\} \mid i \in \{1,\ldots,m\}, j \in \{1,\ldots,n\})\}
\]
and $\Phi = \emptyset$.
\end{prop}

\begin{proof}
The sets $\{c_k, c_\ell \}$, $\{v_i, v_j\}$ are generating for $k \not=\ell$ and nonadjacent $v_i, v_j$ such that $i \not=j$.  So $\{c_i, v_j, v_{j+1} \}$ is a maximal nongenerating set. To see the other containment, consider a maximal nongenerating set $N$. Then $N$ contains at most one central vertex. Since $n\ge4$, any three rim vertices contain two nonadjacent rim vertices. Hence $N$ must contain at most two rim vertices. Thus $N=\{c_i, v_j, v_{j+1}\}$ for some $i$ and $j$. The statement about the Frattini subset is now easy to verify.
\end{proof}

\begin{prop}
For generalized wheel graphs with $n \geq 4$, $\nim(\DNG(W_{m,n}))=1$.
\end{prop}

\begin{proof}
Each set in $\mathcal{N}$ has size $3$ by Proposition~\ref{prop:max-genwheel}, so the result follows from Proposition~\ref{prop:all-same}.
\end{proof}

\begin{prop}
For generalized wheel graphs with $n \geq 4$, $\nim(\GEN(W_{m,n}))=0$.
\end{prop}

\begin{proof}
Every vertex is part of a minimal generating set of the form $\{c_i, c_j\}$ or $\{v_i, v_{i+2}\}$, so the second player wins by selecting a vertex that forms a minimal generating set with the first player's selection. 
\end{proof}

%---------------%
\section{Closing Remarks}
%---------------%

We summarize our results for graph families in Table~\ref{bigTable}.  Recall that trees, complete graphs, and windmill graphs are examples of block graphs.

\begin{table}[h]
\renewcommand{\arraystretch}{1.2}
\begin{tabular}{@{}c|c|c|c@{}}
\noalign{\hrule height 2pt}
$G$ & restriction & $\nim(\DNG(G))$ & $\nim(\GEN(G))$\\
\noalign{\hrule height 1pt}
$K_{m}+\overline{K}_{n}$ & $m\ge1$, $n\ge2$ & $1$ & $0$\\
\noalign{\hrule height 1pt}
$H\circ K_{1}$ & $H$ nontrivial & $1-\pty(V)$ & $\pty(V)$\\
\noalign{\hrule height 1pt}
block &  & $1-\pty(V)$ & $\pty(V)$\\
\noalign{\hrule height 1pt}
\multirow{2}{*}{$C_{n}$} & $3\le n\equiv_{4}1,2$ & $1$ & \multirow{2}{*}{$\pty(V)$}\\
\cline{2-3} %\cline{3-3}
 & $3\le n\equiv_{4}0,3$ & $0$ & \\
\noalign{\hrule height 1pt}
$Q_{n}$ & $n\ge2$ & $0$ & $0$\\
\noalign{\hrule height 1pt}
\multirow{2}{*}{$P_{m}\boxprod P_{n}$} & $2=m\le n$ & $0$ & \multirow{2}{*}{$\pty(V)$}\\
\cline{2-3} %\cline{3-3}
 & $3\le m\le n$ & $2\pty(m+n)$ & \\
\noalign{\hrule height 1pt}
\multirow{3}{*}{$K_{m_{1},\ldots,m_{k}}$} & $\lambda\in\{0,1\}$ & $1-\pty(V)$ & $\pty(V)$\\
\cline{2-4} %\cline{3-4} \cline{4-4}
 & $2\le\lambda\equiv_{2}0$ & \multirow{2}{*}{$\pty(\lambda+\sigma)$} & $2\pty(\sigma)$\\
\cline{2-2} \cline{4-4}
 & $2\le\lambda\equiv_{2}1$ &  & $\pty(\sigma)$\\
\noalign{\hrule height 1pt}
\multirow{4}{*}{$W_{m,n}$} & $m=1$, $n=4$ & \multirow{3}{*}{$1-\pty(V)$} & 2\\
\cline{2-2} \cline{4-4}
 & $m=1$, $n\ge4$ &  & \multirow{2}{*}{$\pty(V)$}\\
\cline{2-2}
 & $m\ge2$, $n=3$ &  & \\
\cline{2-4} %\cline{3-4} \cline{4-4}
 & $m\ge2$, $n\ge4$ & $1$ & $0$\\
\noalign{\hrule height 1.5pt}
\end{tabular}
\vspace{5mm}

\caption{
\label{bigTable} 
Summary of nim-values for graph families.  
}
\end{table}

We propose the following open problems.

\begin{figure}[h]
\begin{tikzpicture}[yscale=.8,xscale=1.1]
\node[vert] (8) at (0,0) {\scriptsize$v_8$};
\node[vert] (3) at (0,3) {\scriptsize$v_3$};
\node[vert] (5) at (.55,1) {\scriptsize$v_5$};
\node[vert] (7) at (1,2) {\scriptsize$v_7$};
\node[vert] (2) at (3,1.5) {\scriptsize$v_2$};
\node[vert,fill=green] (9) at (5,1.5) {\scriptsize$v_9$};
\node[vert] (4) at (6,0) {\scriptsize$v_4$};
\node[vert] (6) at (3,3) {\scriptsize$v_6$};
\node[vert] (1) at (6,3) {\scriptsize$v_1$};
\path [b] (8) to (3) to (7) to (5) to (8);
\path [b] (7) to (2) to (5);
\path [b] (2) to (8) to (9);
\path [b] (9) to (4) to (8);
\path [b] (2) to (9) to (3) to (6) to (9);
\path [b] (6) to (1) to (9);
\path [b] (1) to (4);
\end{tikzpicture}

\caption{
\label{Fig:nim7}
A planar graph $G$ with $\nim(\DNG(G))=7$.}
\end{figure}

\begin{enumerate}
\item
Is it possible to generalize our results about grid graphs to all lattice graphs?

\item The graph in Example~\ref{ex:4-pan} is a pseudotree, a connected graph that has at most one cycle. Is it possible to characterize the nim-values of pseudotrees?

\item
From computer calculations, we have examples of graphs with corresponding nim-values up to $7$ for both games. 
Figure~\ref{Fig:nim7} shows an example of a graph $G$ with $\nim(\DNG(G))=7$.
We conjecture that the spectrum of nim-values for both building games is the set of nonnegative integers. Is this true even for the family of diameter 2 graphs?
\item 
Based on computer calculations, the spectrum of nim-values for both building games seems to be finite for outerplanar graphs and geodetic graphs, where a graph is \emph{geodetic} if it is connected and there is a unique shortest
path between any two vertices. What exactly are these spectrums?
% $\{0,1,2,3\}$ for DNG Not high confidence.
% $\{0,1,2\}$ for GEN Not high confidence.
% Spectrum of $\DNG$ and $\GEN$ for geodetic graphs seems to be $\{0,1,2,3\}$. Spectrum of $\DNG$ and $\GEN$ for diameter 2 geodetic graphs seems to be $\{0,1\}$ and $\{0,1,2\}$, respectively. 
\item Is there a relationship between the Frattini subset and eccentricity approximating spanning trees~\cite{DRAGAN2017142}?
\item 
Is it possible to characterize the Frattini subsets for some graph families like the geodetic graphs?
%Is it related to cut sets? Is it the set of cut vertices for geodetic graphs?
\item 
Is it possible to classify connected graphs with empty Frattini subset?
\item 
Can we say anything about the nim-value of games played on random graphs?  
%possible that for most graphs the nim-value is the same, perhaps 0.
\item 
In the destroying games of~\cite{HararyConnect,SiebenHypergraph}, the players remove elements until the remaining elements no longer generate. What are the nim-values of these destroying convex hull games on graphs?
%\item Can we obtain analogous results for the corresponding destroying games as in~\cite{HararyConnect, SiebenHypergraph}?
\item  Computer calculations suggest that if $\nim(\DNG(G))=0=\nim(\DNG(H))$ and $\pty(V(G))=1=\pty(V(H))$, then $\nim(\DNG(G \cupdot H))=1$. Is this analogue of Proposition~\ref{prop:sum-result} true?
\end{enumerate}

%\newold{}{
%\section{Various things, most likely won't make it into the paper}

% Maybe add this to the generic results section.

%The next conjecture is mainly a guess based on data.

%\begin{conjecture}
%If $\nim(\DNG(G))=0=\nim(\DNG(H))$ and $\pty(G)=1=\pty(H)$,
%then $\nim(\DNG(G+ H))=1$.
%\end{conjecture}

%\begin{proof}
%The second player can win $\nim(\DNG(G+ H))+*1$. 

%Not clear what to do if the first player picks the stone early. The second player need to pick a vertex that does not hurt the winning prospects. It's not clear if this exists.

%Not sure about this any more: The result is not true for hypergraph games. For example if $V=\{1,2,3\}$ and $\mathcal{N}=\{\{1,2\},\{1,3\},\{2,3\}\}$, then $\nim(\DNG(G))=0=?=\nim(\DNG(G+G))$. Perhaps this kind of $\mathcal{N}$ is not possible for convex hulls. This happens for the complement of $K_3$.
%\end{proof}
%}

%---------------%
\section*{Acknowledgements}
%---------------%

This material is based upon work supported by the National Science Foundation under Grant No.~DMS-1929284 while the authors were in residence at the Institute for Computational and Experimental Research in Mathematics in Providence, RI, via the Collaborate@ICERM program.

%---------------%
\bibliographystyle{plain}
\bibliography{game}

\begin{thebibliography}{10}

\bibitem{albert2007lessons}
Michael Albert, Richard Nowakowski, and David Wolfe.
\newblock {\em Lessons in {P}lay: {A}n {I}ntroduction to {C}ombinatorial {G}ame
  {T}heory}.
\newblock CRC Press, 2007.

\bibitem{anderson.harary:achievement}
Marlo Anderson and Frank Harary.
\newblock Achievement and avoidance games for generating abelian groups.
\newblock {\em Internat. J. Game Theory}, 16(4):321--325, 1987.

\bibitem{BeneshErnstSiebenSymAlt}
Bret~J. Benesh, Dana~C. Ernst, and N\'{a}ndor Sieben.
\newblock Impartial avoidance and achievement games for generating symmetric
  and alternating groups.
\newblock {\em Int. Electron. J. Algebra}, 20:70--85, 2016.

\bibitem{BeneshErnstSiebenDNG}
Bret~J. Benesh, Dana~C. Ernst, and N\'{a}ndor Sieben.
\newblock Impartial avoidance games for generating finite groups.
\newblock {\em North-West. Eur. J. Math.}, 2:83--103, 2016.

\bibitem{BeneshErnstSiebenGeneralizedDihedral}
Bret~J. Benesh, Dana~C. Ernst, and N\'{a}ndor Sieben.
\newblock Impartial achievement games for generating generalized dihedral
  groups.
\newblock {\em Australas. J. Combin.}, 68:371--384, 2017.

\bibitem{BeneshErnstSiebenGENSpectrum}
Bret~J. Benesh, Dana~C. Ernst, and N\'{a}ndor Sieben.
\newblock The spectrum of nim-values for achievement games for generating
  finite groups.
\newblock {\em Integers}, 2023.
\newblock To appear.

\bibitem{brandenburg:algebraicGames}
Martin Brandenburg.
\newblock Algebraic games playing with groups and rings.
\newblock {\em Internat. J. of Game Theory}, pages 1--34, 2017.

\bibitem{BuckleyHarary85}
Fred {Buckley} and Frank {Harary}.
\newblock {Closed geodetic games for graphs}.
\newblock {Combinatorics, graph theory and computing, Proc. 16th Southeast.
  Conf., Boca Raton/Fla. 1985, Congr. Numerantium 47, 131--138}, 1985.

\bibitem{BuckleyHarary86}
Fred {Buckley} and Frank {Harary}.
\newblock {Geodetic games for graphs}.
\newblock {\em {Quaest. Math.}}, 8:321--334, 1986.

\bibitem{DRAGAN2017142}
Feodor~F. Dragan, Ekkehard K\"ohler, and Hend Alrasheed.
\newblock Eccentricity approximating trees.
\newblock {\em Discrete Appl. Math.}, 232:142--156, 2017.

\bibitem{ErnstSieben}
Dana~C. Ernst and N\'{a}ndor Sieben.
\newblock Impartial achievement and avoidance games for generating finite
  groups.
\newblock {\em Internat. J. Game Theory}, 47(2):509--542, 2018.

\bibitem{FraenkelHarary89}
Aviezri {Fraenkel} and Frank {Harary}.
\newblock {Geodetic contraction games on graphs}.
\newblock {\em {Int. J. Game Theory}}, 18(3):327--338, 1989.

\bibitem{HARARY1984323}
Frank Harary.
\newblock Convexity in graphs: Achievement and avoidance games.
\newblock In M.~Rosenfeld and J.~Zaks, editors, {\em Annals of Discrete
  Mathematics (20): Convexity and Graph Theory}, volume~87 of {\em
  North-Holland Mathematics Studies}, page 323. North-Holland, 1984.

\bibitem{HararyConnect}
Frank Harary and Robert~W. Robinson.
\newblock Connect-it games.
\newblock {\em Coll. Math. J.}, 15(5):411--419, 1984.

\bibitem{HaynesHenningTiller}
Teresa~W. {Haynes}, Michael~A. {Henning}, and Charlotte {Tiller}.
\newblock {Geodetic achievement and avoidance games for graphs}.
\newblock {\em {Quaest. Math.}}, 26(4):389--397, 2003.

\bibitem{MccoySieben}
Stephanie {McCoy} and N\'andor {Sieben}.
\newblock {Impartial achievement games on convex geometries}.
\newblock {\em {Comput. Geom.}}, 98:15, 2021.
\newblock Id/No 101786.

\bibitem{Necascova}
Milena {Ne\v{c}\'askov\'a}.
\newblock {A note on the achievement geodetic games}.
\newblock {\em {Quaest. Math.}}, 12(1):115--119, 1989.

\bibitem{PelayoBook}
Ignacio~M. Pelayo.
\newblock {\em Geodesic {C}onvexity in {G}raphs}.
\newblock SpringerBriefs Math. New York, NY: Springer, 2013.

\bibitem{Schmidt}
J\"{u}rgen Schmidt.
\newblock Einige grundlegende {B}egriffe und {S}\"{a}tze aus der {T}heorie der
  {H}\"{u}llenoperatoren.
\newblock In {\em Bericht \"{u}ber die {M}athematiker-{T}agung in {B}erlin,
  {J}anuar, 1953}, pages 21--48. Deutscher Verlag der Wissenschaften, Berlin,
  1953.

\bibitem{SiebenHypergraph}
N{\'a}ndor Sieben.
\newblock Impartial hypergraph games.
\newblock {\em Electronic Journal of Combinatorics}, 30(2):P2.13, 1--36, 2023.

\bibitem{SiegelBook}
Aaron~N. Siegel.
\newblock {\em Combinatorial {G}ame {T}heory}, volume 146 of {\em Graduate
  Studies in Mathematics}.
\newblock American Mathematical Society, Providence, RI, 2013.

\bibitem{Wang17}
Yue-Li {Wang}.
\newblock {Geodetic contraction games on trees}.
\newblock In {\em Frontiers in algorithmics. 11th international workshop, FAW
  2017, Chengdu, China, June 23--25, 2017. Proceedings}, pages 233--240. Cham:
  Springer, 2017.

\end{thebibliography}
%---------------%

\end{document}